\documentclass{amsart}

\usepackage{microtype}
\DisableLigatures[f]{encoding = *, family = * }

\usepackage{graphicx}%
\usepackage{multirow}%
\usepackage{amsmath,amssymb,amsfonts}%
\usepackage{amsthm}%
\usepackage{mathrsfs}%
\usepackage[title]{appendix}%
\usepackage{xcolor}%
\usepackage{textcomp}%
\usepackage{manyfoot}%
\usepackage{algorithm}%
\usepackage{algorithmicx}%
\usepackage{algpseudocode}%
\usepackage{listings}%

\usepackage{tabularx}

\usepackage{tikz}

\usepackage{pgfplots}
\pgfplotsset{compat=1.18}
\usepackage{pgfplotstable}
\usetikzlibrary{through,calc,arrows,intersections,pgfplots.groupplots}

\usepackage{bm}

\usepackage{verbatim,booktabs}

\usepackage{hyperref}
\ifpdf
\hypersetup{
	pdftitle={Numerical analysis of a spherical harmonic discontinuous Galerkin method for scaled radiative transfer equations with isotropic scattering},
	pdfauthor={Q. Sheng, C. Hauck, and Y. Xing}
}
\fi
\usepackage{cleveref}

\newtheorem{theorem}{Theorem}[section]
\newtheorem{corollary}[theorem]{Corollary}%
\newtheorem{lemma}[theorem]{Lemma}%
\crefname{lemma}{Lemma}{Lemma}
\newtheorem{prob}{Problem}%
\crefname{prob}{Problem}{Problem}
\newtheorem{assumption}{Assumption}

\newtheorem{remark}{Remark}[section]%

\newcommand{\ud}{\,\mathrm{d}}
\newcommand{\bx}{\mathbf{x}}
\newcommand{\bn}{\mathbf{n}}
\newcommand{\bomega}{\bm{\omega}}

\newcommand{\veps}{\varepsilon}

\newcommand*{\tran}{\mathsf{T}}

\DeclareMathOperator{\diag}{diag}

\begin{document}
	
\title[Uniform convergence of $P_N$ DG schemes]{Numerical analysis of a spherical harmonic discontinuous Galerkin method for scaled radiative transfer equations with isotropic scattering}	

\author{Qiwei Sheng}
\address{Department of Mathematics, California State University, Bakersfield, CA 93311}
\email{qsheng@csub.edu}

\author{Cory Hauck}
\address{Computer Science and Mathematics Division, Oak Ridge National Laboratory, Oak Ridge, TN 37831}
\email{hauckc@ornl.gov}

\author{Yulong Xing}
\address{Department of Mathematics, The Ohio State University, Columbus, OH 43210}
\email{xing.205@osu.edu}

\thanks{
	This work of the second author is supported by the DOE Office of Advanced Scientific Computing Research and by the National Science Foundation under Grant No. 1217170. ORNL is operated by UT-Battelle, LLC., for the U.S. Department of Energy under Contract DE-AC05-00OR22725. The United States Government retains and the publisher, by accepting the article for publication, acknowledges that the United States Government retains a non-exclusive, paid-up, irrevocable, world-wide license to publish or reproduce the published form of this manuscript, or allow others to do so, for the United States Government purposes. The Department of Energy will provide public access to these results of federally sponsored research in accordance with the DOE Public Access Plan (http://energy.gov/downloads/doe-public-access-plan).
	The work of the third author is partially supported by the NSF grant DMS-1753581. }

\begin{abstract}
In highly diffusion regimes when the mean free path $\varepsilon$ tends to zero, the radiative transfer equation has an asymptotic behavior which is governed by a diffusion equation and the corresponding boundary condition. Generally, a numerical scheme for solving this problem has the truncation error containing an $\varepsilon^{-1}$ contribution, that leads to a nonuniform convergence for small $\varepsilon$. Such phenomenons require high resolutions of discretizations, which degrades the performance of the numerical scheme in the diffusion limit. In this paper, we first provide a--priori estimates for the scaled spherical harmonic ($P_N$) radiative transfer equation. Then we present an error analysis for the spherical harmonic discontinuous Galerkin (DG) method of the scaled radiative transfer equation showing that, under some mild assumptions, its solutions converge uniformly in $\varepsilon$ to the solution of the scaled radiative transfer equation. We further present an optimal convergence result for the DG method with the upwind flux on Cartesian grids. Error estimates of $\left(1+\mathcal{O}(\varepsilon)\right)h^{k+1}$ (where $h$ is the maximum element length) are obtained when tensor product polynomials of degree at most $k$ are used. 
\end{abstract}



\maketitle

\section{Introduction}
The radiative transfer equation (RTE) is a type of kinetic equation that models the scattering and absorption of radiation propagating through a material medium. In this paper, we consider the scaled, steady-state, linear RTE with isotropic scattering and periodic boundary conditions for the angular flux $u$ that depends on $\bx \in X\subset\mathbb{R}^d$ and $\bomega \in \mathbb{S}$, the unit sphere in $\mathbb{R}^d$.  For simplicity we assume $X=(0,1)^d$, and for physical problems, $d=3$.  However, if $u$ possesses special symmetries, reduced equations in one or two-dimensional spatial domains can be derived \cite{LM1984}. 

We denote by $\partial X$ the boundary of $X$ and by $\mathbf{k}(\bx)$ the unit outward normal to $X$ at $\bx \in  \partial X$.  We then set $\Gamma^-=\{(\bx,\bomega)\in \partial X \times \mathbb{S} \colon \bomega \cdot \mathbf{k}(\bx) < 0\}$.  With this notation, the RTE with periodic boundary conditions takes the form
\begin{subequations}\label{eq:rte_scale_all}
	\begin{alignat}{2}
		\bomega\cdot\nabla u(\bx,\bomega)+\frac{\sigma_{\mathrm{t}}}{\varepsilon}u(\bx,\bomega)
    &=\left(\frac{\sigma_{\mathrm{t}}}{\varepsilon}-\varepsilon\sigma_{\mathrm{a}}\right)\overline{u}(\bx)+\varepsilon f, &&\quad (\bx,\bomega) \in X \times \mathbb{S},\label{eq:rte_scale}\\
		u(\bx,\bomega)&=u(\bx - \mathbf{k}(\bx),\bomega),&& \quad (\bx,\bomega) \in \Gamma^-,
    \label{eq:rte_scale_pbc}
	\end{alignat}
\end{subequations}
where the scalar flux $\overline{u}=\frac{1}{\int_{\mathbb{S}}\ud\bomega}\int_\mathbb{S} u \ud{\bomega}$ is the average of $u$ over $\mathbb{S}$, 
the functions $\sigma_{\mathrm{a}} = \sigma_{\mathrm{a}}(\bx)$ and $\sigma_{\mathrm{t}}=\sigma_{\mathrm{t}}(\bx)$ are (known) absorption and total cross sections, respectively, and $f=f(\bx,\bomega)$ is a (known) source. The parameter $\varepsilon>0$ is a scaling parameter that measures the relative strength of scattering.   We assume throughout the paper that $0 < \veps < 1$.  Although $u$ depends on $\varepsilon$, we suppress this dependence in the notation for ease of presentation. 

We assume that  
\begin{subequations}\label{as:rte_scale}
	\begin{align}	&\sigma_{\mathrm{t}},\sigma_{\mathrm{a}}\in L^\infty(X),\quad  \sigma_{\mathrm{t}}(\bx)-\varepsilon^2\sigma_{\mathrm{a}}(\bx) > 0  ~\text{a.e.}~ \mathbf{x} \in X,\label{as:rte_scale_1}\\
	&\sigma_{\mathrm{t}}>\sigma_{\mathrm{a}} \ge \sigma_{\mathrm{a}}^{\mathrm{min}} \text{ in } X \text{   for some constant } \sigma_{\mathrm{a}}^{\mathrm{min}} >0,\label{as:rte_scale_2}\\
		&f(\bx,\bm{\omega})\in L^2(X\times\mathbb{S}).\label{as:rte_scale_3}
	\end{align}
\end{subequations}
The quantity $\sigma_{\mathrm{s},\varepsilon} = \sigma_{\mathrm{t}}(\bx)-\varepsilon^2\sigma_{\mathrm{a}}(\bx)$ is a  (non-dimensional) scattering cross-section. While the condition \eqref{as:rte_scale_2} is not strictly necessary, it is frequently used in a priori estimates. 
Under the assumptions in \eqref{as:rte_scale}, the system \eqref{eq:rte_scale_all} has a unique solution \cite{agoshkov1998boundary} in 
{$\mathring{H}^1_2(X\times\mathbb{S})$}---the subspace of  $H^1_2(X\times\mathbb{S}):=\{v \in L^2(X\times \mathbb{S})\colon {\bomega}\cdot \nabla v \in L^2(X\times \mathbb{S})\}$ containing all functions that are periodic in each spatial component $x_i$, $i \in \{1,\cdots,d\}$, with norm
\begin{equation*}
    \|v\|_{\mathring{H}^1_2(X\times\mathbb{S})}=\|v\|_{H^1_2(X\times\mathbb{S})}=(\|v\|^2_{L^2(X\times\mathbb{S})}+\|\bm{\omega}\cdot \nabla v\|^2_{L^2(X\times\mathbb{S})})^{1/2}.
\end{equation*}

If $\varepsilon \ll 1$ and $\sigma_{\mathrm{t}} / \varepsilon \gg 1$ uniformly across the spatial domain $X$, then solution $u$ of the RTE can be approximated accurately by the solution of a diffusion equation that is much cheaper to compute \cite{habetler1975uniform,LarsenKeller,bensoussan1979boundary}. However, for true multiscale applications, the size of $\sigma_{\mathrm{t}}$ varies significantly as a function of $\bx$, so that both diffusive and non-diffusive regions may coexist. In such situations, it is necessary to utilize kinetic models like the RTE for accurate simulations.  Unfortunately, traditional numerical approaches for solving the RTE may be highly inefficient efficient in diffusive regions. In particular, local truncation errors may scale like $h^p / \varepsilon$, where $h>0$ is the spatial mesh size and $p>0$ is an integer corresponding to the order of the method \cite{jin2010asymptotic,lowrie2002methods}. In such cases, accuracy may degrade dramatically as the parameter $\veps \to 0$. For this reason, asymptotic preserving (AP) schemes \cite{Jin1999, jin2010asymptotic} have been developed.  Such schemes transition to consistent and stable numerical schemes for the macroscopic model as $\varepsilon \to 0$. 

Numerical discretization of the RTE is often carried out separately for the angular and spatial variables.  In this paper, we consider angular discretizations using spherical harmonics expansions. The spherical harmonics ($P_N$) method \cite{CZ1967,davison1957neutron} is one of the most widely-employed angular discretization methods for the RTE.  It is also used to discretize kinetic semiconductor equations \cite{rupp2016review}, which share many features with the RTE.
The $P_N$ method is a class of spectral methods that approximates the solution to the RTE by a linear combination of spherical harmonics basis functions in angle.  In the steady-state case, the coefficients of this expansion are functions of space that satisfy a symmetrizable, linear hyperbolic system of equations.  Due to the properties of the spherical harmonics, the matrices in this system are very sparse and thereby cheap to assemble and apply. 

The $P_N$ method enjoys several beneficial properties.  Among them: (i) it preserves the rotational symmetry of the transport equation \cite{garrett2016eigenstructure}; (ii) for steady-state problems, it is equivalent to the diffusion limit when $N=1$; and (iii) for sufficiently smooth solutions, the convergence of the method is spectral \cite{FHK2016}. 
Moreover, under certain conditions, the $P_N$ solution converges to the solution of the RTE at a rate $\mathcal{O}(\varepsilon^N)$ as $\varepsilon \to 0$, and its angular moments (which correspond to physically meaningful quantities) may converge even faster \cite{CH2019}. 

While the current paper focuses on periodic conditions, inflow boundary conditions that prescribe known data on $\Gamma^-$ are more physically relevant.
The treatment of such boundary conditions is possible with direct approaches \cite{bunger2020stable} and variational frameworks \cite{ES2012,MRS1999}.  For vacuum (zero inflow) conditions, an extended computational domain can be employed to maintain sparsity of the $P_N$ system \cite{SW2021, egger2019perfectly} or to handle complicated geometries \cite{powell2019pseudospectral} in a periodic setting.

Spatial discretizations strategies for the $P_N$ equations include finite differences \cite{seibold2014starmap}, least squares \cite{brown2003moment,varin2005spherical} and mixed finite elements \cite{egger2019perfectly}, discontinuous Galerkin methods \cite{SW2021}, and psuedo-spectral methods \cite{powell2019pseudospectral}.  Methods for capturing the diffusion limit for the time-dependent $P_N$ equations \cite{hauck2010methods} have been explored in both finite-volume \cite{hauck2009temporal} and discontinuous Galerkin (DG) \cite{mcclarren2008effects,mcclarren2008semi} contexts; and a rigorous analysis of the steady-state problem, based on a least-squares finite element formulation, can be found in \cite{manteuffel1999boundary}.

In the current paper, we focus on DG discretizations of the $P_N$ equations ($P_N$-DG). DG methods were first introduced in \cite{reed1973triangular} to simulate transport equations like \eqref{eq:rte_scale_all}, and a rigorous analysis was carried out in \cite{LR1974} for a simplified equation without scattering. 
A complete space-angle convergence analysis of discrete ordinate DG methods was conducted in \cite{HHE2010} for problems with scattering, and analogous results for the spherical harmonic DG method were recently proved in \cite{SW2021}. While the analysis studied in \cite{HHE2010,SW2021} is valid for any fixed $\veps = \mathcal{O}(1)$, it does not address the behavior of the DG method when $\veps \to 0$.
In the context of discrete ordinates angular discretization, it is known that DG methods can capture the diffusion limit whenever the underlying approximation space supports globally continuous linear polynomials \cite{LM1989,Adams2001,GK2010}.  This requirement translates to local $\mathbb{P}_1$ and $\mathbb{Q}_1$ approximations for triangular and rectangular elements, respectively. In \cite{jin2010asymptotic}, a general framework was provided to construct uniform error estimates from conventional error bounds and asymptotic analysis, although these estimates are typically not sharp.  Sharper estimates for discrete-ordinate DG discretizations that are uniform in $\varepsilon$ were recently established in \cite{SH2021}.

In this paper, we investigate the convergence of $P_N$-DG methods for the scaled RTE with isotropic scattering and periodic boundary conditions.  The analysis builds upon ideas from \cite{SW2021} for establishing angular error estimates and from \cite{SH2021} for establishing spatial error estimates. The main contribution is to derive spectral estimates for the $P_N$ discretization that are uniform in $\veps$ as well as $\mathcal{O}(h^k)$ error estimates for the DG scheme that are uniform in $\veps$ whenever the DG space contains $\mathbb{Q}_1$ for Cartesian cells and $\mathbb{P}_1$ for triangles. Under this assumption on $k$, the DG estimate improves upon the general strategy in \cite{jin2010asymptotic} (which leads to an error bound of $\mathcal{O}(h^{k/2 + 1/4})$) and it implies that the $P_N$-DG scheme is necessarily accurate in the diffusion limit.  Another contribution of this paper is that, to the best of the authors' knowledge, we prove for the first time a uniform optimal error estimate $\mathcal{O}(h^{k+1})$ for the spatial discretization on Cartesian meshes with tensor product elements.   

The analysis presented here does have some limitations.  For example, the uniform error estimates may come at a price: the error obtained in \cite{SW2021} is $\mathcal{O}(h^{k+1/2})$ for a fixed $\veps$, while here we can obtain a uniform $\mathcal{O}(h^{k})$ bound across all  $\veps \in (0,1)$. In addition,  the assumptions in \eqref{as:rte_scale} imply that voids (i.e. $\sigma_{\rm{t}}=0$) are not allowed and, moreover, that $\sigma_{\rm{t}} / \veps \gg 1$ whenever $\veps \ll 1$.  Hence as $\veps \to 0$, the problem in \eqref{eq:rte_scale_all} becomes uniformly diffusive.  In other words, regions of very thick and thin materials do not coexist at the same time, as is often the case in realistic problems.  Therefore the analysis here should be considered as a first step in analyzing more realistic problems.

The rest of this article is organized as follows. In \Cref{sec:settings}, the scaled RTE with isotropic scattering and its spherical harmonic discretization are introduced. We also present the notations to be used in the remainder of this paper, and the a priori estimates regarding the solutions of the spherical harmonic equation. In the last part of this section, we define the DG scheme for the spherical harmonic equation and prove its stability and well-posedness. In \Cref{sec:err_ana}, uniform convergence and error estimates of the spherical harmonic DG method for the scaled RTE are established.  In \Cref{sec:err_ana_optimal}, we prove an optimal error estimate for the upwind DG scheme in one-dimensional slab geometries and on multidimensional Cartesian mesh with tensor product elements. Conclusions are given in \Cref{sec:conclusion}.

\section{Preliminaries and problem setting}\label{sec:settings}
We start with presenting some notations. Throughout the paper, the symbol $a\lesssim b$ abbreviates $a\le Cb$ for any two real quantities $a$ and $b$, with $C>0$ being a nonessential constant independent of the finite element mesh size, which may take different values at different appearances. 
The conventional notation $H^r(D)$ was adopted to indicate Sobolev spaces on (possibly lower-dimensional) subdomain $D\subset X$ with the norm $\|\cdot\|_{r,D}$. Clearly, we have $H^0(D)=L^2(D)$ whose norm is denoted by $\|\cdot\|_D$. 

\subsection{Spherical harmonic method for angular discretization}\label{sec:SH_PN}
In this section, the spherical harmonic functions and the $P_N$ method will be briefly reviewed, we refer the interested readers to, e.g., \cite{AH2012} and \cite{Claus1966}. The angle bracket is introduced as a short-hand notation for the angular integration over $\mathbb{S}$: $\langle \cdot\rangle = \int_\mathbb{S} (\cdot) \ud \bm{\omega}$.

For any $\bomega\in \mathbb{S}$, let $\bomega=\begin{bmatrix} \omega_1 & \omega_2 &\omega_3 \end{bmatrix}^{\tran} = \begin{bmatrix} \sqrt{1-\mu^2}\cos(\varphi) & \sqrt{1-\mu^2}\sin(\varphi) & \mu \end{bmatrix}^{\tran}$, where $\mu:=\omega_3\in [-1,1]$ and $\varphi\in [0,2\pi)$ is the angle between the $x_1$-axis and the projection of $\bomega$ onto the $x_1$-$x_2$ plane, respectively. 
Given integers $\ell\ge 0$ and $\kappa\in [-\ell,\ell]$, the normalized, real-valued spherical harmonic of degree $\ell$ and order $\kappa$ is expressed in terms of $\mu$ and $\varphi$ as $m^\kappa_\ell(\bomega)=\alpha_\ell^\kappa P^\kappa_\ell(\mu)\, T^\kappa(\varphi)$, where $P^\kappa_\ell$ is an associated Legendre function, $T^\kappa$ is a sinusoidal function \cite{FHK2016}, and $\alpha_\ell^\kappa = \sqrt{\frac{(2\ell+1)}{4 \pi} \frac{(\ell - |\kappa|)!}{(\ell+|\kappa|)!}}$ is a constant.

We collect the $n_\ell := 2\ell+1$ real-valued normalized harmonics of degree $\ell$ together into a vector-valued function $\bm{m}_\ell=\begin{bmatrix} m^{-\ell}_\ell & m^{-\ell+1}_\ell & \cdots & m^0_\ell & \cdots & m^{\ell-1}_\ell & m^\ell_\ell \end{bmatrix}^{\tran}$ and for any given $N$, we set $\bm{m} = \begin{bmatrix}\bm{m}_0^{\tran} & \bm{m}_1^{\tran} & \cdots & \bm{m}_N^{\tran}\end{bmatrix}^{\tran}$. Note that the normality of $\bm{m}$ is not necessary but adopting it would bring great convenience in the later deductions. The vector $\bm{m}$ has $L := \sum^N_{\ell=0} n_\ell= (N+1)^2$ components which form an orthogonal basis for the space
$\mathbb{P}_N=\left\{\sum^N_{\ell=0} \sum^\ell_{\kappa=-\ell} c^\kappa_\ell m^\kappa_\ell : c^\kappa_\ell\in \mathbb{R} \text{ or } L^2(X), \text{ and } 0\le\ell\le N, |\kappa|\le\ell\right\}$.
Furthermore, the spherical harmonics fulfill a recursion relation of the form \cite{AH2012}
\begin{equation}\label{eq:SH_rec}
\omega_i \bm{m}_\ell = \bm{A}^{(i)}_{\ell,\ell+1}\bm{m}_{\ell+1} + \bm{A}^{(i)}_{\ell,\ell-1} \bm{m}_{\ell-1},
\end{equation}
where $\bm{A}^{(i)}_{\ell,\ell'} = \langle\omega_i\bm{m}_\ell\bm{m}^{\tran}_{\ell'}\rangle$ and $\big(\bm{A}^{(i)}_{\ell,\ell'}\big)^{\tran}=\bm{A}^{(i)}_{\ell',\ell}$.

The $P_N$ equations for the RTE \eqref{eq:rte_scale_all} can be derived by firstly approximating $u(\bx,\bomega)$ with a function $u_{P_N} \in \mathbb{P}_N$ in the form of
\begin{equation}\label{eq:u_appx}
u_{P_N}(\bx,\bomega) := \bm{m}^{\tran}(\bomega)\bm{u}(\bx),
\end{equation}
such that, for any $v\in \mathbb{P}_N$,
\begin{equation}\label{eq:rte_scale_pre_Pn}
\left\langle v\left(\bomega\cdot\nabla u_{P_N}(\bx,\bomega)+\frac{\sigma_{\mathrm{t}}(\bx)}{\varepsilon} u_{P_N}(\bx,\bomega) -\left(\frac{\sigma_{\mathrm{t}}(\bx)}{\varepsilon}-\varepsilon\sigma_{\mathrm{a}}(\bx)\right)\bar{u}_{P_N}\right) \right\rangle = \varepsilon\langle v f \rangle. 
\end{equation} 
We denote 
$u_\ell^k := \int_{\mathbb{S}}m_\ell^k u_{P_N} \ud\bomega,  \quad \ell=0,\cdots,N, \; |k|\le\ell$, and $\bm{u}_\ell = \begin{bmatrix}u_\ell^{-\ell}&\cdots&u_\ell^\ell\end{bmatrix}^{\tran}$. Then $\bm{u}= \begin{bmatrix}\bm{u}_0^{\tran} & \bm{u}_1^{\tran} & \cdots & \bm{u}_N^{\tran} \end{bmatrix}^{\tran}$. We also use a single index to denote the components in $\bm{m}$ and $\bm{u}$, e.g., $\bm{u}=\begin{bmatrix}u_1 & u_2 & \cdots & u_L\end{bmatrix}^{\tran}$. The function $u_{P_N}$ can be expressed as
\begin{equation*}
u_{P_N} = \bm{u}^{\tran}\bm{m} = \sum_{\ell=0}^N\bm{u}^{\tran}_\ell\bm{m}_\ell = \sum_{\ell=0}^N\sum_{|k|\leq \ell} {u}^k_\ell{m}_\ell^k=\sum_{i=1}^{L}m_iu_i.
\end{equation*}
Setting $v=\bm{m}$ in \eqref{eq:rte_scale_pre_Pn} and using \eqref{eq:u_appx} and the fact that $\langle \bm{m}\bm{m}^{\tran}\rangle=\bm{I}$, we can reformulate \eqref{eq:rte_scale_pre_Pn} into a system of hyperbolic equations ($P_N$ equations): 
\begin{subequations}\label{eq:rte_scale_Pn_all}
	\begin{align}
		\bm{A}\cdot \nabla \bm{u}(\bx) + \varepsilon\sigma_{\mathrm{a}}\bm{u}(\bx) + \left(\frac{\sigma_{\mathrm{t}}}{\varepsilon}-\varepsilon\sigma_{\mathrm{a}}\right) \bm{R}\bm{u}(\bx) &= \varepsilon\bm{f}(\bx) \qquad \text{ in } X, \label{eq:rte_scale_Pn}\\
		\bm{u}(\bx)&=\bm{u}(\bx+\mathbf{k}) \quad \text{ on } \partial X, \label{eq:rte_scale_Pn_bc}
	\end{align}
\end{subequations}
for arbitrary $\bx$ and certain $\mathbf{k}$ such that if $\bx\in \partial X$, then $\bx+\mathbf{k}\in \partial X$. 
Here, $\bm{u}(\bx)\in \mathbb{R}^L$,
$\bm{f}=\langle\bm{m}f\rangle$, $\bm{R}=\bm{I}-\diag(1,0,\cdots,0)=\diag(0,1,\cdots,1)$  is diagonal and positive semi-definite,
and the dot product between $\bm{A}:=\begin{bmatrix} \left(\bm{A}^{(1)}\right)^\tran & \left(\bm{A}^{(2)}\right)^\tran & \left(\bm{A}^{(3)}\right)^\tran \end{bmatrix}^\tran$ and the gradient is understood as
$\bm{A}\cdot \nabla := \sum^3_{i=1}\bm{A}^{(i)}\partial_i$ with $\bm{A}^{(i)}=\int_{\mathbb{S}}\omega_i\bm{m}\bm{m}^{\tran}\ud\bomega$,  $i=1,2,3$. Note that $\bm{A}^{(i)},i=1,2,3$ is symmetric and sparse. Indeed, the recursion relation \eqref{eq:SH_rec} and orthogonality conditions for the associated Legendre functions \cite{AH2012, Claus1966} imply that $\bm{A}^{(i)}_{\ell,\ell'}$ is nonzero only if $\ell' = \ell \pm 1$. Therefore
\begin{equation}\label{eq:A_struc}
\bm{A}^{(i)} = 
\begin{bmatrix}
0 				& \bm{A}^{(i)}_{0,1} 	& 0             	& 0 				& \dots					& 0  \\
\bm{A}^{(i)}_{1,0} & 0 			  	& \bm{A}^{(i)}_{1,2}	& 0 				& \dots 			& 0\\
0				& \bm{A}^{(i)}_{2,1}  	& 0 			   	& \bm{A}^{(i)}_{2,3}   & \dots				& 0\\
0				& 0					& \ddots			& \ddots			& \ddots			& 0\\
\vdots			& \vdots			& \ddots 			&\bm{A}^{(i)}_{N-1,N-2}& 0					&\bm{A}^{(i)}_{N-1,N} \\
0				& 0					& \dots				&0					&\bm{A}^{(i)}_{N,N-1}& 0 \\
\end{bmatrix}.
\end{equation}
Moreover, since $\bm{A}^{(i)}$, $i=1,2,3$ are symmetric, they can be diagonalized as:
\begin{equation}\label{eq:A_sym}
\bm{A}^{(i)}=\mathcal{Q}_i\varLambda^{(i)} \mathcal{Q}_{i}^{\tran},
\end{equation}
where $\mathcal{Q}_i$ is a real orthogonal matrix and $\varLambda^{(i)}=\diag(\lambda^{(i)}_1,\lambda^{(i)}_2,\cdots,\lambda^{(i)}_L)$ is real.
Let $|\varLambda^{(i)}|:=\diag(|\lambda^{(i)}_1|,|\lambda^{(i)}_2|,\cdots,|\lambda^{(i)}_L|)$, then we define $|\bm{A}^{(i)}|=\mathcal{Q}_i |\varLambda^{(i)}| \mathcal{Q}_i^{\tran}$, which will be used for the definition of numerical flux later. 

Let $D\subseteq X$ be a (possibly lower-dimensional) subdomain of $X$. Given $\bm{u}, \bm{v}\in \left[L^2(D)\right]^L$,  define the inner product
\begin{equation}\label{eq:vector_inner}
(\bm{u},\bm{v})_D=\int_D \bm{u}^{\tran} \bm{v}\ud\bx = \sum_{\ell=0}^N\sum_{|k|\leq \ell} \int_D u^k_\ell\,v^k_\ell \ud \bx,
\end{equation}
as well as the norms 
\begin{equation*}\label{eq:norm_def}
\|\bm{u}\|_{r,D}=\left(\sum_{\ell=0}^N\sum_{|k|\leq \ell}\|u^k_\ell\|_{r,D}^2\right)^{1/2} \quad\text{ for } \bm{u}\in \left[H^r(D)\right]^L.
\end{equation*}
When $r=0$ and $D=X$, we omit the subscripts $0$ and $X$, i.e., $(\bm{u},\bm{v})=(\bm{u},\bm{v})_X$ and $\|\bm{u}\|:=\|\bm{u}\|_{0,X}$. Since $\langle \bm{m}\bm{m}^{\tran}\rangle=\bm{I}$, it follows that
\begin{equation*}
\|u_{P_N}\|_{L^2(X\times\mathbb{S})} =\left(\int_X\langle (\bm{m}^{\tran}\bm{u})^2\rangle \ud \bx\right)^{1/2}=\left(\int_X \bm{u}^{\tran}\langle\bm{m}\bm{m}^{\tran}\rangle\bm{u} \ud \bx\right)^{1/2}=\|\bm{u}\|.
\end{equation*}
Define the space
$\bm{V}:=\left\{\bm{u}\in [L^2(X)]^{L}: \bm{A}\cdot\nabla\bm{u}\in [L^2(X)]^{L}\right\}$
with the associated norm 
$\|\bm{u}\|_{\bm{V}} := \|\bm{u}\| + \|\bm{A}\cdot\nabla\bm{u}\|$,
and the space $\bm{W}:=\left\{\bm{u}\in\bm{V}:\bm{u}\right.$ is $1$-periodic in each spatial argument $x_i$, $\left.i=1,2,3\right\}$. We have the following result \cite[Theorem 6]{SW2021}.
\begin{theorem}\label{th:Pn_wellposed}
	Assume that \eqref{as:rte_scale} holds. Then, for any fixed $\varepsilon>0$, the system of $P_N$ equations \eqref{eq:rte_scale_Pn_all} has a unique solution $\bm{u}\in\bm{W}$.
\end{theorem}

The space $\bm{V}$ has a well-defined trace, and the following integration by parts result  \cite[Corollary B.57]{EG2004} holds: $\forall \bm{u},\bm{v}\in \bm{V}$,
\begin{equation}\label{eq:int_by_part}
	(\bm{A}\cdot\nabla\bm{u}, \bm{v})_{D} = (\bn\cdot\bm{A}\bm{u}, \bm{v})_{\partial D} - \sum_{i=1}^3(\bm{A}^{(i)}\bm{u},\partial_i \bm{v})_D,
\end{equation} 
for any Lipschitz domain $D\subseteq X$, where $\bn\cdot\bm{A} = \sum_{i=1}^3 n_i\bm{A}^{(i)}$. Letting $\bm{v}=\mathbf{e}_i$, $i=1,\cdots,L$ in \eqref{eq:int_by_part}, where $\mathbf{e}_i$ is the $i$th column of the identity matrix, yields the divergence formula:
\begin{equation*}\label{eq:div}
	\int_D \bm{A}\cdot\nabla\bm{u} \ud \bx = \int_{\partial D}\bn\cdot\bm{A}\bm{u}\ud \bx.
\end{equation*} 
The following lemma describes the continuity of the functions in $\bm{V}$. 
\begin{lemma}[{\cite[Lemma 2]{SW2021}}]\label{lem:cont}
	Assume $\bm{u}\in\bm{V}$. Then for any Lipschitz surface $E\subset X$, there holds 
	\begin{equation*}\label{eq:cont}
		\int_{E} \bn\cdot\bm{A}\big(\bm{u}|_{D_1}-\bm{u}|_{D_2}\big)\ud\bx=\mathbf{0},
	\end{equation*}
	where $D_1$ and $D_2$ are subdomains of a Lipschitz domain $D\subseteq X$ such that $D=D_1\cup D_2$ with $D_1\cap D_2=\emptyset$ and $\overline{D}_1\cap\overline{D}_2=E$, i.e.,  $E$ is a shared surface of $D_1$ and $D_2$. 
\end{lemma}
Since $E\subset X$ is arbitrary, by \Cref{lem:cont}, $\bn\cdot\bm{A}\big(\bm{u}|_{D_1}-\bm{u}|_{D_2}\big)=\mathbf{0}$ almost everywhere on any surface in $X$. 

\subsection{Variational Formulation}
Multiplying \eqref{eq:rte_scale_Pn_all} by an arbitrary function $\bm{v}\in \left[H^1_2(X)\right]^L$, integrating over $X$, using integration-by-parts, and employing the periodic boundary condition \eqref{eq:rte_scale_Pn_bc}, we get a variational formulation of \eqref{eq:rte_scale_Pn_all}:
\begin{equation}\label{eq:rte_var_form}
\mathfrak{a}(\bm{u},\bm{v})=\ell(\bm{v}),\quad \forall \bm{v}\in \left[H^1_2(X)\right]^L,
\end{equation}
where
\begin{equation*} 
\mathfrak{a}(\bm{u},\bm{v})= (\bm{A}\cdot \nabla\bm{u},\bm{v}) + (\bm{Q}\bm{u},\bm{v}) \;\text{ and }\;
\ell(\bm{v})= \varepsilon (\bm{f},\bm{v}). 
\end{equation*}
Here
\begin{equation}\label{eq:def_Q}
	\bm{Q} = \varepsilon\sigma_{\mathrm{a}}\bm{I} + \left(\frac{\sigma_{\mathrm{t}}}{\varepsilon}-\varepsilon\sigma_{\mathrm{a}}\right) \bm{R} = \begin{bmatrix}\varepsilon\sigma_{\mathrm{a}} &  \\  & \frac{\sigma_{\mathrm{t}}}{\varepsilon}\bm{I}_{L-1} \end{bmatrix}, 
\end{equation}
where $I$ is the identity matrix. 
We also define $\sqrt{\bm{Q}} = \diag(\sqrt{\varepsilon\sigma_{\mathrm{a}}}, \;  \sqrt{\frac{\sigma_{\mathrm{t}}}{\varepsilon}}\bm{I}_{L-1})$,  
and denote by 
\begin{equation}\label{eq:Q_norm}
	(\bm{u},\bm{v})_{\bm{Q}}=(\bm{Q}\bm{u},\bm{v}) \quad\text{ and } \quad \|\bm{u}\|_{\bm{Q}}=(\bm{u},\bm{u})_{\bm{Q}}^{1/2}=\|\sqrt{\bm{Q}}\bm{u}\|.
\end{equation}
By re-scaling $\bm{u}_0$ and $\bm{u}_\ell$, $\ell=1,\cdots, N$ with $\sqrt{\varepsilon}$ and $1/\sqrt{\varepsilon}$, respectively, we have the following lemma on the relation between estimates in terms of $\|\cdot\|$ and $\|\cdot\|_{\bm{Q}}$.
\begin{lemma}\label{lem:norms_rescale}
	For $\|\cdot\|_{\bm{Q}}$ defined in \eqref{eq:Q_norm} with $\bm{Q}$ defined in \eqref{eq:def_Q}, we have
	\begin{equation*}
		\|\bm{u}\|\lesssim \left(\sqrt{\varepsilon}+\frac{1}{\sqrt{\varepsilon}}\right)\|\bm{u}\|_{\bm{Q}}.
	\end{equation*}
\end{lemma}

Since $\bm{Q}$ is symmetric and strictly positive definite, $(\cdot, \cdot)_{\bm{Q}}$ is an inner product. 
Due to the periodic boundary condition, a direct calculation shows that
\begin{equation}\label{eq:a_stab}
	\mathfrak{a}(\bm{u},\bm{u}) = (\bm{Q}\bm{u},\bm{u}). 
\end{equation}

\subsection{A priori Estimates}

The purpose of this subsection is to derive some a priori estimates for the spherical harmonic equation \eqref{eq:rte_scale_Pn_all}.

\begin{lemma}\label{lem:pri_1}
	Assume that \eqref{as:rte_scale} holds. Let $\bm{u}=\begin{bmatrix}u_1 & u_2 & \cdots & u_{L}\end{bmatrix}^{\tran}$ be the solution of \eqref{eq:rte_scale_Pn_all}. Then the following estimates hold: 
	\begin{equation*}\label{lem:priori_est_1}
		\|u_1\| = \mathcal{O}(1) \text{ and } \|u_i\| = \mathcal{O}(\varepsilon), \; i=2,3,\cdots,L.
	\end{equation*}
	Furthermore, we have
	\[ \|\bm{u}\|=\mathcal{O}(1+\varepsilon). \]
\end{lemma}
\begin{proof}
	Taking $\bm{v}=\bm{u}$ in \eqref{eq:rte_var_form} and noting \eqref{eq:a_stab}, we have
	\begin{equation}\label{eq:est_iden_1}
		(\bm{Q}\bm{u},\bm{u}) = \varepsilon (\bm{f},\bm{u}).
	\end{equation}
	Applying the Cauchy–Schwarz inequality to \eqref{eq:est_iden_1} yields 
	\begin{multline*}
		\varepsilon(\sigma_{\mathrm{a}}u_1,u_1)+\frac{1}{\varepsilon}\sum_{i=2}^{L}(\sigma_{\mathrm{t}}u_i,u_i)
		=\varepsilon\sum_{i=1}^{L}(f_i,u_i)\\
		\le \varepsilon(\|f_1/\sqrt{\sigma_{\mathrm{a}}}\|^2/2+\|\sqrt{\sigma_{\mathrm{a}}} u_1\|^2/2) + \varepsilon\sum_{i=2}^{L}\left(\varepsilon^2\|f_i/\sqrt{\sigma_{\mathrm{t}}}\|^2/2+\frac{\|\sqrt{\sigma_{\mathrm{t}}}u_i\|^2}{2\varepsilon^2}\right),
	\end{multline*}
	from which we obtain
	\[ \frac{\varepsilon}{2}(\sigma_{\mathrm{a}}u_1,u_1)+\frac{1}{2\varepsilon}\sum_{i=2}^{L}(\sigma_{\mathrm{t}}u_i,u_i) \le \frac{\varepsilon \|f_1/\sqrt{\sigma_{\mathrm{a}}}\|^2}{2} +  \varepsilon^3\sum_{i=2}^{L}\|f_i/\sqrt{\sigma_{\mathrm{t}}}\|^2/2, \]
	i.e.,
	\begin{equation*}
		\left\{\begin{aligned}
			\varepsilon (\sigma_{\mathrm{a}}u_1,u_1)&\le \varepsilon \|f_1/\sqrt{\sigma_{\mathrm{a}}}\|^2 +  \varepsilon^3\sum_{i=2}^{L}\|f_i/\sqrt{\sigma_{\mathrm{t}}}\|^2,\\
			\frac{1}{\varepsilon}(\sigma_{\mathrm{t}}u_i,u_i)&\le \varepsilon \|f_1/\sqrt{\sigma_{\mathrm{a}}}\|^2 +  \varepsilon^3\sum_{i=2}^{L}\|f_i/\sqrt{\sigma_{\mathrm{t}}}\|^2,\quad i=2,3,\cdots,L.
		\end{aligned}\right.
	\end{equation*}
	We deduce that 
	\begin{equation}\label{eq:priori_est_1_entry}
		\|u_1\| \lesssim \|f_1\|+\varepsilon\sum_{i=2}^{L}\|f_i\|,\quad \|u_i\| \lesssim \varepsilon\|f_1\|+\varepsilon^2\sum_{i=2}^{L}\|f_i\|, \; i=2,3,\cdots,L,
	\end{equation}
	which proves the lemma.
\end{proof}

\subsection{Discontinuous {G}alerkin method for spatial discretization}\label{sec:DG}
Equation~\eqref{eq:rte_scale_Pn_all} is a first-order symmetric hyperbolic system in space, which will be discretized by the DG method in this section. 

Let $\mathcal{T}_h$ be a regular family of partition of the domain $X=(0,1)^3$ with elements $K$.
To avoid unnecessary technicalities, we assume the meshes on the boundary to be periodic, i.e., the surface meshes on two opposite parallel faces of $\partial X$ are identical.
Let $\mathcal{E}^{\mathrm{int}}_h$ be the set of all interior faces of $\mathcal{T}_h$ which includes all inner surfaces in each $K$.
Define $h_K:= \textrm{diam}(K)$ and $h:=\max_{K\in \mathcal{T}_h}h_K$.
Denote by $\bn_K=\begin{bmatrix}n_1^K & n_2^K & n_3^K\end{bmatrix}^{\tran}$ the unit outward normal to $\partial K$ with respect to the element $K$,
and $e$ the generic interface of a mesh element $K\in \mathcal{T}_h$.
For an interface $e\subset \partial K$, we use $\bn_e$ to denote the unit outward normal to $e$ with respect to $K$.

Let  $V_{h}^k$ be a discontinuous Galerkin finite element space whose elements are polynomials of degree no more than $k$ when restricted to any element $K$, i.e., 
\begin{equation}\label{eq:poly_space}
	V_h^k=\begin{cases}
		\{v\in L^2(\mathbb{S}); \;v|_K\in \mathbb{P}_k(K)\;\forall K\in \mathcal{T}_h\}, &\text{ if } \mathcal{T}_h \text{ is triangular/tetrahedral}, \\
		\{v\in L^2(\mathbb{S}); \;v|_K\in \mathbb{Q}_k(K)\;\forall K\in \mathcal{T}_h\}, &\text{ if } \mathcal{T}_h \text{ is rectangular/cuboidal},
	\end{cases}
\end{equation}
where $k$ is a nonnegative integer, $\mathbb{P}_k(K)$ denotes the set of all polynomials on $K$ of a total degree no more than $k$, and $\mathbb{Q}_k(K) = \{\sum_j c_j p_j(x)q_j(y)r_j(z): p_j,q_j,r_j$ polynomials of degree $\le k\}$. Define $\bm{V}_{h}^\ell=[V_{h}^k]^{2\ell+1}$, $\ell=0,1,\cdots,N$,
and the Cartesian product $\bm{V}_h= \bm{V}_{h}^0\times \bm{V}_{h}^1\times \cdots \times \bm{V}_{h}^N = [V_{h}^k]^{L}$. 

Denote $E+F:=\{u+v:\, u\in E, \, v\in F\}$, for any two function spaces $E$ and $F$. For any given element $K\in\mathcal{T}_h\subset \mathbb{R}^3$, and any function $\bm{f}\in \bm{V}_h + [C(\overline{X})]^L$, we define the inside and outside values of $\bm{f}$ (with respect to $K$) on the interface $\partial K$ as, respectively,
\begin{subequations}
	\begin{align*}
	\bm{f}^- (\bx) &=\lim\limits_{\epsilon \to 0^+}\bm{f}(\bx- \epsilon\,\bn_K(\bx)), \quad \;\;\;\forall \bx\in \partial K,\\
	\bm{f}^+ (\bx) &=\begin{cases}
	\displaystyle\lim_{\epsilon \to 0^+}\bm{f}(\bx + \epsilon\,\bn_K(\bx)), & \forall \bx\in \partial K\backslash\partial X,\\
	\bm{f}^- (\tilde{\bx}), & \forall \bx\in \partial K\cap\partial X,
	\end{cases}
	\end{align*}
\end{subequations}
where $\tilde{\bx}$ is a point on $\partial X$ corresponding to $\bx$ associated with the periodic boundary condition. Note that the extensions of the outside values on $\partial X$ reflect the periodic boundary conditions. The jump in $\bm{f}\in \bm{V}_h$ with respect to $K$ is $[\![\bm{f}]\!]=\bm{f}^+-\bm{f}^-$, and the average of $\bm{f}$ is denoted by
$\{\!\!\{\bm{f}\}\!\!\}=(\bm{f}^+ + \bm{f}^-)/2$.

For any $\bm{u}_h,\bm{v}_h\in \bm{V}_h + [C(\overline{X})]^L$, we define
\begin{subequations}
	\begin{align}\label{eq:a}
	\mathfrak{a}_h(\bm{u}_h,\bm{v}_h) &= \sum_{K\in\mathcal{T}_h}\bigg\{ \Big(\bn_K\cdot\overrightarrow{\bm{A}\bm{u}_h}, \, \bm{v}_h^-\Big)_{\partial K} - \sum_{i=1}^3(\bm{A}^{(i)}\bm{u}_h, \partial_i \bm{v}_h)_K + (\bm{Q}\bm{u}_h,\bm{v}_h)_K \bigg\}, \\
	\mathfrak{f}(\bm{v}_h) &= \varepsilon\sum_{K\in\mathcal{T}_h}(\bm{f},\bm{v}_h)_K.
	\end{align}
\end{subequations}
Here, following \cite{HLL1983}, we define the numerical flux $\bn_K\cdot\overrightarrow{\bm{A}\bm{u}_h}$ as
\begin{equation}\label{eq:flux_Pn}
	\bn_K\cdot\overrightarrow{\bm{A}\bm{u}_h} = \bn_K\cdot\bm{A}\{\!\!\{\bm{u}_h\}\!\!\} -\frac{1}{2}|\bn_K|\cdot\bm{D}\,[\![\bm{u}_h]\!],	
\end{equation}
where $|\bn_K|\cdot\bm{D} = \sum_{i=1}^3 |n_i^K|\bm{D}^{(i)}$ and, assuming the upwind flux, $\bm{D}^{(i)}=|\bm{A}^{(i)}|$. Applying integration by parts to \eqref{eq:a} leads to an equivalent form of the bilinear form $\mathfrak{a}_h$, which will be useful later:
\begin{equation*}\label{eq:def_a_2}
\mathfrak{a}_h(\bm{u}_h,\bm{v}_h) = \sum_{K\in\mathcal{T}_h}\bigg\{(\bm{A}\cdot\nabla\bm{u}_h, \bm{v}_h)_K - (\bm{u}_h^-, \bn_K\cdot\overleftarrow{\bm{A}\bm{v}_h})_{\partial K} + (\bm{Q}\bm{u}_h,\bm{v}_h)_K\bigg\},
\end{equation*}
where the downwind trace $\bn_K\cdot\overleftarrow{\bm{A}\bm{v}_h}$ is defined by
\begin{equation*}
\bn_K\cdot\overleftarrow{\bm{A}\bm{v}_h} = \bn_K\cdot\bm{A}\{\!\!\{\bm{v}_h\}\!\!\} +\frac{1}{2}|\bn_K|\cdot\bm{D}\,[\![\bm{v}_h]\!].
\end{equation*}
We can present the fully discrete spherical harmonic DG scheme for \eqref{eq:rte_scale_all} as follows.
\begin{prob}\label{prob:SH_DG}
	Find $\bm{u}_h\in \bm{V}_h$ such that
	\begin{equation}\label{eq:bilinear}
	\mathfrak{a}_h(\bm{u}_h,\bm{v}_h) = \mathfrak{f}(\bm{v}_h) \quad \forall \bm{v}_h\in \bm{V}_h.
	\end{equation}
\end{prob}
\begin{remark}
	Note that the periodic boundary condition \eqref{eq:rte_scale_Pn_bc} is not imposed into the definition of the DG finite element space. Instead, we use it implicitly in the definition of the numerical flux \eqref{eq:flux_Pn} during the construction of the spherical harmonics DG method.
\end{remark}
\begin{remark}\label{rmk:a_ext}
	Following \Cref{lem:cont}, it is easy to validate that the discrete problem \eqref{eq:bilinear} is consistent, i.e., given the exact solution $\bm{u}$ of the $P_N$ model \eqref{eq:rte_scale_Pn_all}, we have $\mathfrak{a}_h(\bm{u},\bm{v}_h) = \mathfrak{f}(\bm{v}_h)$ for any $\bm{v}_h\in \bm{V}_h$, and therefore, the following Galerkin orthogonality holds
	\begin{equation}\label{eq:go}
	\mathfrak{a}_h(\bm{u}-\bm{u}_h,\bm{v}_h) = 0 \quad \forall \bm{v}_h\in \bm{V}_h.
	\end{equation}
\end{remark}

To study the stability and well-posedness of \Cref{prob:SH_DG}, we define the following norm: 
\begin{equation}\label{def:norm}
|||\bm{v}_h|||_h = \left(\frac{1}{4}\sum_{K\in\mathcal{T}_h} \Big(|\bn_K|\cdot\bm{D}\,[\![\bm{v}_h]\!], [\![\bm{v}_h]\!]\Big)_{\partial K} + (\bm{Q}\bm{v}_h,\bm{v}_h)\right)^\frac{1}{2}, \quad \forall \bm{v}_h\in \bm{V}_{h}.
\end{equation}
Note that, for $i=1,2,3$, $0\le |n_i^K|\le 1$ and $\|\bn_K\|=1$. Since  $\bm{D}^{(i)}$ are positive semi-definite and $\bm{Q}$ is positive definite for any $\varepsilon>0$, we verify that $|||\cdot|||_h$ is a norm.

\begin{lemma}[Stability]\label{lem:stab}
	Under conditions \eqref{as:rte_scale_1} and \eqref{as:rte_scale_2}, we have
	\begin{equation}\label{eq:stab}
	|||\bm{v}_h|||^2_h = \mathfrak{a}_h(\bm{v}_h,\bm{v}_h), \quad \forall \bm{v}_h\in \bm{V}_h.
	\end{equation}
\end{lemma}
\begin{proof}
	For any $\bm{v}_h\in \bm{V}_h$, we have, by the definition \eqref{eq:a}, that $\mathfrak{a}_h(\bm{v}_h,\bm{v}_h) = \mathrm{I} +\Pi + \mathrm{III}$, where
	\begin{align*}
	\mathrm{I} &:=  \sum_{K\in\mathcal{T}_h} \sum_{i=1}^3 \Big(n_i^K\bm{A}^{(i)}\{\!\!\{\bm{v}_h\}\!\!\} -\frac{1}{2}|n_i^K|\bm{D}^{(i)}[\![\bm{v}_h]\!], \bm{v}^{-}_h\Big)_{\partial K}, \\
	\mathrm{II} &:=  -\sum_{K\in\mathcal{T}_h} \sum_{i=1}^3 (\bm{A}^{(i)}\bm{v}_h,\partial_i \bm{v}_h)_K, \quad
	\mathrm{III} :=  \sum_{K\in\mathcal{T}_h} (\bm{Q}\bm{v}_h,\bm{v}_h)_K.
	\end{align*}
	For the first term, since $A$ is symmetric, we can simplify it as	\begin{equation*}
	\mathrm{I}
	=  \sum_{K\in\mathcal{T}_h}  \frac{1}{2}\Big(\bn_K\cdot\bm{A} \bm{v}^{-}_h, \bm{v}^{-}_h\Big)_{\partial K} + \sum_{K\in\mathcal{T}_h} \frac{1}{4} \Big(|\bn_K|\cdot\bm{D}\,[\![\bm{v}_h]\!], [\![\bm{v}_h]\!]\Big)_{\partial K}. 
	\end{equation*}
	Since $(\bm{A}^{(i)}\bm{v}_h,\partial_i \bm{v}_h)_K=(\varLambda^{(i)} \mathcal{Q}_i^{\tran}\bm{v}_h,\partial_i \mathcal{Q}_i^{\tran}\bm{v}_h)_K$, the second term can be handled as:
	\begin{equation*}
	\mathrm{II}
	= -\sum_{K\in\mathcal{T}_h} \sum_{i=1}^3 \frac{1}{2} (n_i^K\varLambda^{(i)} \mathcal{Q}_i^{\tran}\bm{v}^{-}_h, \mathcal{Q}_i^{\tran}\bm{v}^{-}_h)_{\partial K} 
	= -\sum_{K\in\mathcal{T}_h} \frac{1}{2} (\bn_K\cdot\bm{A}\bm{v}^{-}_h, \bm{v}^{-}_h)_{\partial K}. 
	\end{equation*}
	Therefore, we have $\mathrm{I} + \mathrm{II} = \sum_{K\in\mathcal{T}_h} \frac{1}{4} \Big(|\bn_K|\cdot\bm{D}\,[\![\bm{v}_h]\!], [\![\bm{v}_h]\!]\Big)_{\partial K}$, and \eqref{eq:stab} follows.
\end{proof}

All norms in the finite-dimensional space $\bm{V}_h$ should be equivalent, which ensures the continuity of the bilinear form. The following corollary is a direct consequence of the above result, due to the well-known Lax–Milgram lemma.
\begin{corollary}
	Under the assumptions \eqref{as:rte_scale_1} and \eqref{as:rte_scale_2}, the spherical harmonic DG method \eqref{eq:bilinear} has a unique solution.
\end{corollary}

\section{Error analysis}\label{sec:err_ana}
In this section, we provide the error analysis of the SH-DG approximation for solving the RTE \eqref{eq:rte_scale_all} with periodic boundary conditions. Let $u$ and $\bm{u}_h$ be the solutions of \eqref{eq:rte_scale_all} and \eqref{eq:bilinear}, respectively. The total numerical error $u-\bm{m}^{\tran}\bm{u}_h$ can be divided into two parts:
\begin{equation*}
	u-\bm{m}^{\tran}\bm{u}_h = (u- \bm{m}^{\tran}\bm{u}) + (\bm{m}^{\tran}\bm{u}-\bm{m}^{\tran}\bm{u}_h).
\end{equation*}
Since $\bm{m}^{\tran}\bm{u}=u_{P_N}$ in \eqref{eq:u_appx}, the first part stands for the error due to the angular discretization, while the second part accounts for the error of the spatial discretization. 

\subsection{Error analysis for the angular discretization}
We start with unraveling the angular discretization error $\tau(\bx,\bomega) := u(\bx,\bomega) - u_{P_N}(\bx,\bomega)$ arising from the semi-discretization \eqref{eq:rte_scale_pre_Pn}.
Define $\mathcal{P}_N u = \bm{m}^{\tran}\langle \bm{m}u\rangle$, then $\mathcal{P}_N$ is the $L^2$-orthogonal projection of a generic function $u$ on $\mathbb{S}$ onto $\mathbb{P}_N$. The following two lemmas will be useful in the proof of the main result. 
\begin{lemma}\label{lem:interp_ang} \cite[Theorem 2.1]{STW2011}
	Assume $v\in H^q(\mathbb{S})$ for some $q>0$, then
	\begin{equation*}
	\|v-\mathcal{P}_N v\|_{L^2(\mathbb{S})} \lesssim N^{-q}\|v\|_{H^q(\mathbb{S})}.
	\end{equation*}
\end{lemma}
\begin{lemma}\label{lem:uni_bnd_A}\cite{FHK2016}
	The matrix $\bm{A}^{(i)}_{\ell,\ell+1}$, $i=1,2,3$ defined in \eqref{eq:SH_rec} is uniformly bounded in the induced $2$-norm.
\end{lemma}
The following bound for $\bm{A}^{(i)}$, $i=1,2,3$ follows from the recursion relation \eqref{eq:SH_rec}.
\begin{corollary}\label{cor:uni_bnd_A}
	The matrix $\bm{A}^{(i)}$, $i=1,2,3$ is uniformly bounded in the induced $2$-norm.
\end{corollary}

The main result is the following theorem on the error of angular discretization. The proof follows \cite{SW2021} with additional attention given to the parameter $\varepsilon$. 
\begin{theorem}\label{thm:ang_est}
	Let $u_{P_N}=\bm{m}^{\tran}\bm{u}$ be the solution to \eqref{eq:rte_scale_Pn_all}. Assume the solution $u$ to \eqref{eq:rte_scale_all} satisfies the additional regularity assumptions
	\begin{equation}\label{as:reg_Omega}
	\text{for some } q>0, \,  u\in L^2(X; H^q(\mathbb{S})), \; \partial_i u\in L^2(X; H^q(\mathbb{S})),\; i=1,2,3.
	\end{equation}
	Then
	\begin{equation*}
	\| u- u_{P_N}\|_{L^2(X\times\mathbb{S})}\lesssim N^{-q} \left( \|u\|_{L^2(X; H^q(\mathbb{S}))} + \sum_{i=1}^{3}\|\partial_i u\|_{L^2(X; H^q(\mathbb{S}))} \right).
	\end{equation*}
\end{theorem}
\begin{proof}
    We split the error $\tau = u - u_{P_N}$ into the projection error $\eta=u-\mathcal{P}_N u$ and a remainder $\xi=\mathcal{P}_N u - u_{P_N} = \bm{m}^{\tran}(\langle \bm{m}u\rangle - \bm{u})$ which is an element of $\mathbb{P}_N$, that is, $	\tau = \eta + \xi$. Therefore,
	\begin{equation}\label{eq:ang_err_split}
	\|\tau\|_{L^2(X\times\mathbb{S})} \le \|\eta\|_{L^2(X\times\mathbb{S})} + \|\xi\|_{L^2(X\times\mathbb{S})}.
	\end{equation}
	By the assumption \eqref{as:reg_Omega} and \Cref{lem:interp_ang}, we have
	\begin{equation}\label{eq:ang_intp}
	\|\eta\|_{L^2(X\times\mathbb{S})}=\| u-\mathcal{P}_N u \|_{L^2(X\times\mathbb{S})} \lesssim N^{-q}\|u\|_{L^2(X; H^q(\mathbb{S}))}.
	\end{equation}
	
	The next step is to estimate $\xi$. First, it follows from \eqref{eq:rte_scale} and \eqref{eq:rte_scale_pre_Pn} that, for all $v\in \mathbb{P}_N$,
	\begin{equation}\label{eq:error_total}
	\left\langle\left(\bm{\omega}\cdot \nabla \tau +\frac{\sigma_{\mathrm{t}}}{\varepsilon}\tau\right)v\right\rangle
	= \left\langle\left(\frac{\sigma_{\mathrm{t}}}{\varepsilon}-\varepsilon\sigma_{\mathrm{a}}\right)v\bar{\tau}\right\rangle.
	\end{equation}
	Denote $\bm{\xi}=\langle \bm{m}u\rangle - \bm{u}$, then $\xi = \bm{m}^{\tran}\bm{\xi}$.
	Taking $v=\xi$ in \eqref{eq:error_total} gives
	\begin{multline}\label{eq:error_eq}
	\langle \xi(\bm{\omega}\cdot \nabla \eta)\rangle + \frac{1}{2}\langle \bm{\omega}\cdot \nabla (\xi)^2\rangle + \frac{\sigma_{\mathrm{t}}}{\varepsilon}\langle\xi^2\rangle + \frac{\sigma_{\mathrm{t}}}{\varepsilon}\langle\eta\xi\rangle \\
	= \left(\frac{\sigma_{\mathrm{t}}}{\varepsilon}-\varepsilon\sigma_{\mathrm{a}}\right)\langle \bm{\xi}^{\tran}\bm{m}\, \bar{\eta}\rangle + \left(\frac{\sigma_{\mathrm{t}}}{\varepsilon}-\varepsilon\sigma_{\mathrm{a}}\right)\langle \bm{\xi}^{\tran}\bm{m}\, \overline{\bm{m}^{\tran}\bm{\xi}}\rangle.
	\end{multline}
	By the recursion relation \eqref{eq:SH_rec} of the spherical harmonics, 
	\begin{multline*}
	\langle \xi(\bm{\omega}\cdot \nabla \eta)\rangle = \langle \bm{\xi}^{\tran}\bm{m}\,(\bm{\omega}\cdot \nabla \eta)\rangle = \bm{\xi}^{\tran}  \nabla\cdot\langle\bm{\omega}\, \bm{m}\eta\rangle = \bm{\xi}^{\tran} \nabla\cdot\langle  \bm{A}\, \bm{m}\eta\rangle \\
	= \bm{\xi}^{\tran} \nabla\cdot \langle \bm{A}\, \bm{m} (u-\mathcal{P}_N u)\rangle
	= \bm{\xi}^{\tran}_N \bm{A}_{N,N+1}\cdot \nabla \langle \bm{m}_{N+1} u\rangle.
	\end{multline*}
	Since $\mathcal{P}_N$ is an orthogonal projection on $\mathbb{S}$, $\sigma_{\mathrm{t}}\langle\eta\xi\rangle=0$, and
	\begin{equation*}
	\langle \bm{\xi}^{\tran}\bm{m}\, \bar{\eta}\rangle =  \bm{\xi}^{\tran}\langle\bm{m}\, \bar{\eta}\rangle
	= \bm{\xi}^{\tran}(\bm{I}-\bm{R})\langle \bm{m}(u-\mathcal{P}_N u)\rangle =0.
	\end{equation*}
	The last term on the right-hand side of \eqref{eq:error_eq} can be rewritten as
	\begin{equation*}
	\langle \bm{\xi}^{\tran}\bm{m}\, \overline{\bm{m}^{\tran}\bm{\xi}}\rangle =
	\bm{\xi}^{\tran}\langle \bm{m}\, \overline{\bm{m}^{\tran}}\rangle\bm{\xi} =
	\bm{\xi}^{\tran} \langle \bm{m}\, (\bm{I}-\bm{R})\bm{m}^{\tran}\rangle\bm{\xi} = \bm{\xi}^{\tran} (\bm{I}-\bm{R})\bm{\xi}.
	\end{equation*}
	Therefore, \eqref{eq:error_eq} can be reformulated into
	\begin{equation*}
	\bm{\xi}^{\tran}_N \bm{A}_{N,N+1}\cdot \nabla \langle \bm{m}_{N+1} u\rangle + \frac{1}{2}\langle \bm{\omega}\cdot \nabla (\bm{\xi}^{\tran}\bm{\xi})\rangle + \varepsilon\sigma_{\mathrm{a}}\bm{\xi}^{\tran}\bm{\xi} + \left(\frac{\sigma_{\mathrm{t}}}{\varepsilon}-\varepsilon\sigma_{\mathrm{a}}\right)\bm{\xi}^{\tran} \bm{R}\bm{\xi} = 0.
	\end{equation*}
	Integrating the above equation over $X$ and noting that $\xi$ is periodic on $\partial X$, and therefore $\int_{\partial X}\bomega\cdot\bn\, (\bm{\xi}^{\tran}\bm{\xi})\ud\bx\ud\bomega=0$, we have
	\begin{equation*}
	(\bm{A}_{N,N+1}\cdot \nabla \langle \bm{m}_{N+1} u\rangle, \bm{\xi}_N) + (\varepsilon\sigma_{\mathrm{a}}\bm{\xi},\bm{\xi}) + \left(\left(\frac{\sigma_{\mathrm{t}}}{\varepsilon}-\varepsilon\sigma_{\mathrm{a}}\right)\bm{R}\bm{\xi}, \bm{\xi}\right) = 0.
	\end{equation*}
	Since $\bm{R}$ is positive semi-definite, applying the Cauchy–Schwarz inequality yields 
	\begin{equation*}
		(\bm{Q}\bm{\xi},\bm{\xi})\le \|\bm{A}_{N,N+1}\cdot \nabla \langle \bm{m}_{N+1} u\rangle\|\cdot \|\bm{\xi}_N\|\le \frac{\varepsilon}{2\sigma_{\mathrm{t}}}\|\bm{A}_{N,N+1}\cdot \nabla \langle \bm{m}_{N+1} u\rangle\|^2 + \frac{\sigma_{\mathrm{t}}}{2\varepsilon}\|\bm{\xi}_N\|^2.
	\end{equation*}
	Therefore, by the definition \eqref{eq:def_Q} of $\bm{Q}$,
	\begin{equation}\label{eq:err_cls}
		\|\sqrt{\bm{Q}}\bm{\xi}\|  \lesssim \varepsilon\|\bm{A}_{N,N+1}\cdot \nabla \langle \bm{m}_{N+1} u\rangle\|.
	\end{equation} 
	The term on the right-hand side of the above inequality can be considered as the closure error. 
	Using Lemma~\ref{lem:interp_ang} and Lemma~\ref{lem:uni_bnd_A}, the closure error in \eqref{eq:err_cls} can be further estimated as
	\begin{align*}
	\|\bm{A}_{N,N+1}\cdot \nabla \langle \bm{m}_{N+1} u\rangle\|
	&\lesssim \sum_{i=1}^{3}\|\langle \bm{m}_{N+1} \partial_i u\rangle\| = \sum_{i=1}^{3}\|\langle (\mathcal{P}_{N+1}- \mathcal{P}_N) \partial_i u\rangle\|\\
	& \le \sum_{i=1}^{3}\left(\|\langle (\mathcal{I} - \mathcal{P}_{N+1} )\partial_i u\rangle\| + \|\langle (\mathcal{I} - \mathcal{P}_{N} )\partial_i u\rangle\| \right)\\
	& \le  2\sum_{i=1}^{3}\|\langle (\mathcal{I}- \mathcal{P}_N) \partial_i u\rangle\| 
	\lesssim  N^{-q}\sum_{i=1}^{3}\|\partial_i u\|_{L^2(X; H^q(\mathbb{S}))}.
	\end{align*}
	This, together with \eqref{eq:ang_err_split}, \eqref{eq:ang_intp} and \eqref{eq:err_cls}, completes the proof.
\end{proof}

\subsection{Error Analysis for the numerical solution of the spherical harmonic equation}\label{sec:asym_error_analysis}
For the purpose of error analysis, we make the following assumptions on the regularity of $\bm{u}$.
\begin{assumption}\label{assupt:reg_1}
	For some $r>1$, $\bm{u}\in \left[H^r(X)\right]^L$, and furthermore,
	\begin{equation}\label{eq:reg_assump}
		\|u_1\|_{r,X} = \mathcal{O}(1) \text{ and } \|u_i\|_{r,X} = \mathcal{O}(\varepsilon), \; i=2,3,\cdots,L.
	\end{equation}
	In other words, we assume
	\begin{equation}\label{eq:reg_assump_u}
		\|\bm{u}\|_{r,X}=\mathcal{O}(1+\varepsilon).
	\end{equation}
\end{assumption}
\begin{remark}
	From \Cref{lem:pri_1}, we know that the above assumptions are true when $r=0$. Assume $\sigma_{\mathrm{t}}(\bx)$, $\sigma_{\mathrm{s}}(\bx)$, and $f(\bx)$ are smooth enough and $\partial_j\bm{u}(\bx) = \partial_j\bm{u}(\bx+\mathbf{k})$ on $\partial X$. Taking any partial derivative $\partial_j$ ($j=1,2,3$) of \eqref{eq:rte_scale_Pn} with respect to $x$, $y$ or $z$, respectively, we have
	\begin{equation*}
		\bm{A}\cdot \nabla (\partial_j\bm{u}) + \bm{Q}(\partial_j\bm{u}) = \varepsilon\partial_j\bm{f} -\varepsilon(\partial_j\sigma_{\mathrm{a}})\bm{u} - \partial_j\left(\frac{\sigma_{\mathrm{t}}}{\varepsilon}-\varepsilon\sigma_{\mathrm{a}}\right) \bm{R}\bm{u} \quad \text{ in } X.		
	\end{equation*}
	By the estimate \eqref{eq:priori_est_1_entry}, we have 
	\begin{equation*}
		\|\partial_j u_1\|\lesssim \|\partial_j f_1\|+\|u_1\|= \mathcal{O}(1), 
	\end{equation*}
	and
	\begin{equation*}
		\|\partial_j u_i\|\lesssim \varepsilon (\|f_1\|+\|u_1\|) + \varepsilon^2\sum_{i=2}^{L}\left(\|f_i\|+\|u_i\|+\|u_i\|/\varepsilon^2\right)=\mathcal{O}(\varepsilon), \quad i=2,3,\cdots,L.
	\end{equation*}
	Hence, \Cref{lem:pri_1} remains true for $\partial_i\bm{u}$, i.e., \Cref{assupt:reg_1} holds for $r=1$. Results for $r\ge2$ can also be derived analogously.
\end{remark}

We begin with two lemmas which will be used later. The first lemma gives the convergence rate of the interpolation. For any $K\in\mathcal{T}_h$, define $\mathcal{I}_h$ to be an interpolation operator from $L^2(K)$ onto $\mathbb{P}_k(K)$ which satisfies the following requirements: (i) the restriction of $\mathcal{I}_h$ to any face $e\subset K$ is uniquely determined by the interpolation points on $e$; and (ii) two adjacent elements have the same interpolation points on their shared edge/surface. For example, the commonly employed nodal finite elements satisfy such requirements (see, e.g., \cite[Chapter 3]{BS2008} or \cite[Theorem 2.2.1]{C2002}). Then by using the trace theorem and the scaling argument, we can easily obtain the following result (cf. \cite{AH2009,BS2008,C2002}).
\begin{lemma}[Polynomial interpolation error estimate]\label{lem:interp}
	For all $v\in H^r(K)$ with $r>0$ and $K\in \mathcal{T}_h$,
	\begin{equation*}\label{eq:interp_est}
	\|(I-\mathcal{I}_h) v\|_{q,K}\lesssim h_K^{\min \{r,k+1\}-q}\|v\|_{r,K}, \quad \|(I-\mathcal{I}_h) v\|_{0,\partial K}\lesssim h_K^{\min \{r,k+1\}-1/2}\|v\|_{r,K},
	\end{equation*}
	where $k$ is the degree of the polynomial used in \eqref{eq:poly_space} for the definition of $V^k_h$. 
\end{lemma}

The second lemma is similar to the second Strang lemma in the error analysis of nonconforming finite element methods but in the format of vector arguments (cf.~\cite{HHE2010}). 
\begin{lemma}\label{lem:stab_0}
	Let $\{\bm{Z}_h\}_{h>0}:=\{[Z_h]^L\}_{h>0}$ be a family of finite-dimensional product spaces equipped with norms $\{\|\cdot\|_h\}_{h>0}$. Let $\mathfrak{b}_h(\cdot,\cdot)$ be a uniformly coercive bilinear form over $\bm{Z}_h\times \bm{Z}_h$, i.e., there exists a positive constant $\gamma$ independent of $h$ such that
	\begin{equation*}\label{eq:stab_lem}
	\gamma\|\bm{z}_h\|^2_h\le \mathfrak{b}_h(\bm{z}_h,\bm{z}_h), \quad \forall \bm{z}_h\in \bm{Z}_h.
	\end{equation*}
	Let $\bm{Z}=[Z]^L$ be the product space of an (infinite dimensional) function space $Z$ and assume $\bm{v}\in \bm{Z}$ is a vector of functions and $\mathfrak{b}_h$ can be extended to $(\bm{Z}+\bm{Z}_h)\times \bm{Z}_h$ such that $|\mathfrak{b}_h(\bm{v},\bm{z}_h)|<C$ for all $\bm{z}_h\in \bm{Z}_h$, and $\bm{v}_h$ is an element in $\bm{Z}_h$ satisfying
	\begin{equation*}\label{eq:rte_scale_lem_1_asspt2}
	\mathfrak{b}_h(\bm{v}-\bm{v}_h,\bm{z}_h)=0, \quad \forall \bm{z}_h\in \bm{Z}_h.
	\end{equation*}
	Furthermore, norms $\|\cdot\|_h$ can also be defined on $\bm{Z}+\bm{Z}_h$.
	Then
	\begin{equation}\label{eq:error_inq1}
	\|\bm{v}-\bm{v}_h\|_h\le \inf_{\bm{z}_h\in \bm{Z}_h}\left\{\|\bm{v}-\bm{z}_h\|_h + \frac{1}{\gamma}\sup_{\hat{\bm{z}}_h\in \bm{Z}_h}\frac{\mathfrak{b}_h(\bm{v}-\bm{z}_h,\hat{\bm{z}}_h)}{\|\hat{\bm{z}}_h\|_h}\right\}.
	\end{equation}
\end{lemma}

\subsubsection{Error analysis for non-constant elements} 
We now establish the following error bound between the DG approximation $\bm{u}_h$ and the solution $\bm{u}$ to the $P_N$ system \eqref{eq:rte_scale_Pn_all} when $k\ge 1$.
\begin{theorem}\label{thm:X}
	Assume that \eqref{as:rte_scale_1}--\eqref{as:rte_scale_2} and \Cref{assupt:reg_1} hold. Consider only $k\ge 1$. 
	The spherical harmonic DG method \eqref{eq:bilinear} admits the following error estimate:
	\begin{equation}\label{eq:est_total_general}
	\|\bm{u}-\bm{u}_h\|_{\bm{Q}} \lesssim \sqrt{\varepsilon}\Big(h^{\min\{r,k+1\}-1} \big(1+\mathcal{O}(\varepsilon)\big) + h^{\min\{r,k+1\}} \Big).
	\end{equation}
\end{theorem}
\begin{proof}
	 Since $\mathfrak{a}_h(\bm{v}_h,\bm{v}_h) = |||\bm{v}_h|||^2_h\ge \|\bm{v}_h\|^2_{\bm{Q}}$ by \Cref{lem:stab}, the $\bm{Q}$ norm can be employed in \eqref{eq:error_inq1}. Then it follows from the Galerkin orthogonality \eqref{eq:go} and \Cref{lem:stab,lem:stab_0} that
	\begin{equation}\label{eq:error_inq2}
		\|\bm{u}-\bm{u}_h\|_{\bm{Q}} \lesssim \|\bm{u} - \mathcal{I}_h\bm{u}\|_{\bm{Q}} + \sup_{\bm{v}_h\in \bm{V}_h}\frac{\mathfrak{a}_h(\bm{u}-\mathcal{I}_h\bm{u},\bm{v}_h)}{\|\bm{v}_h \|_{\bm{Q}}},
	\end{equation}
	where $\mathcal{I}_h$ denotes the interpolation operator onto $\bm{V}_h$ in an element-wise way, i.e., for $\bm{v}\in L^2(X)$, $\mathcal{I}_h\bm{v}|_K = \mathcal{I}_K(\bm{v})$. 
	By the definition of $\|\cdot\|_{\bm{Q}}$, \Cref{lem:interp}, and \Cref{assupt:reg_1}, we have
	\begin{equation}\label{eq:error_inq_part_one}
	\|\bm{u} - \mathcal{I}_h\bm{u}\|_{\bm{Q}} \lesssim h^{\min\{r,k+1\}} \|\sqrt{\bm{Q}}\bm{u}\|_{r,X} \lesssim \sqrt{\varepsilon} h^{\min\{r,k+1\}}. 
	\end{equation}
	On the other hand, by the definition of $\mathfrak{a}_h(\cdot,\cdot)$, we have
	\begin{equation}\label{eq:error_inq_bi_parts}
	\mathfrak{a}_h(\bm{u}-\mathcal{I}_h\bm{u},\bm{v}_h) = \mathbb{I}_1 +\mathbb{I}_2 + \mathbb{I}_3,
	\end{equation}
	where
	\begin{align*}\label{eq:a_decomp_2} 
		\mathbb{I}_1 &:= -\sum_{K\in\mathcal{T}_h}\sum_{i=1}^3 ((\bm{u}-\mathcal{I}_h\bm{u})^-, n_i^K\overleftarrow{\bm{A}^{(i)}\bm{v}_h})_{\partial K}, \\
		\mathbb{I}_2 &:=  \sum_{K\in\mathcal{T}_h} (\bm{A}\cdot \nabla (\bm{u}-\mathcal{I}_h\bm{u}),\bm{v}_h)_K, \quad
		\mathbb{I}_3 :=  \sum_{K\in\mathcal{T}_h} \left(\bm{Q}(\bm{u}-\mathcal{I}_h\bm{u}),\bm{v}_h\right)_K.
	\end{align*}	
	Owing to \Cref{lem:cont}, $\bn\cdot\bm{A}\,[\![\bm{u}]\!]=\mathbf{0}$ almost everywhere. Furthermore, due to the requirement of the interpolation operator $\mathcal{I}_h$, $\mathcal{I}_h\bm{u}$ is also continuous a.e.~when $k\ge 1$, which implies $[\![\mathcal{I}_h\bm{u}]\!]=0$ a.e.~when $k\ge 1$. Therefore, for the first term, 
	\begin{equation}\label{eq:error_inq_bi_1}
		\mathbb{I}_1 = \frac{1}{2}\sum_{K\in\mathcal{T}_h}  \Big( [\![\bm{u}-\mathcal{I}_h\bm{u}]\!], \bn_K\cdot\overleftarrow{\bm{A}\bm{v}_h}\Big)_{\partial K} = \frac{1}{2}\sum_{K\in\mathcal{T}_h}  \Big( [\![\bm{u}]\!]-[\![\mathcal{I}_h\bm{u}]\!], \bn_K\cdot\overleftarrow{\bm{A}\bm{v}_h}\Big)_{\partial K} = 0. 
	\end{equation}
	To bound $\mathbb{I}_2$, recalling that $\bm{A}^{(i)}$ has the fine structure \eqref{eq:A_struc}, we have
	\begin{align*}
	\mathbb{I}_2 
		&= \sum_{K\in\mathcal{T}_h} \sum_{i=1}^3 \left(\bm{A}^{(i)}\partial_i (\bm{u}-\mathcal{I}_h\bm{u}),\bm{v}_h\right)_K \nonumber\\
		&= \sum_{K\in\mathcal{T}_h} \sum_{i=1}^3 \left(\sqrt{\bm{Q}}^{-1}\begin{bmatrix}\bm{A}^{(i)}_{0,1}\partial_i (\bm{u}_1-\mathcal{I}_h\bm{u}_1)\\ \bm{A}^{(i)}_{1,0}\partial_i (\bm{u}_0-\mathcal{I}_h\bm{u}_0) + \bm{A}^{(i)}_{1,2}\partial_i (\bm{u}_2-\mathcal{I}_h\bm{u}_2) \\ \vdots \\ \bm{A}^{(i)}_{N,N-1}\partial_i (\bm{u}_{N-1}-\mathcal{I}_h\bm{u}_{N-1}) \end{bmatrix},\sqrt{\bm{Q}}\begin{bmatrix}\bm{v}_{0,h}\\\bm{v}_{1,h}\\\vdots\\\bm{v}_{N,h}\end{bmatrix}\right)_K. \nonumber
	\end{align*}
	We use an inverse inequality and the definition of $\bm{Q}$ in \eqref{eq:def_Q}, and proceed as follows:
	\begin{align}\label{eq:error_inq_bi_2}
		|\mathbb{I}_2| &\lesssim \sum_{K\in\mathcal{T}_h} \sum_{i=1}^3 h^{-1}\left\|\begin{bmatrix}\varepsilon^{-1/2}\bm{A}^{(i)}_{0,1} (\bm{u}_1-\mathcal{I}_h\bm{u}_1)\\ \sqrt{\varepsilon}\bm{A}^{(i)}_{1,0} (\bm{u}_0-\mathcal{I}_h\bm{u}_0) + \sqrt{\varepsilon}\bm{A}^{(i)}_{1,2} (\bm{u}_2-\mathcal{I}_h\bm{u}_2) \\ \vdots \\ \sqrt{\varepsilon}\bm{A}^{(i)}_{N,N-1} (\bm{u}_{N-1}-\mathcal{I}_h\bm{u}_{N-1}) \end{bmatrix}\right\|_K \left\|\sqrt{\bm{Q}}\bm{v}_h\right\|_K \nonumber\\
		&\lesssim h^{-1}\sum_{K\in\mathcal{T}_h}  h^{\min\{r,k+1\}} \Big(\varepsilon\|\bm{u}_0\|^2_{r,K} + (\varepsilon+\varepsilon^{-1})\|\bm{u}_1\|^2_{r,K} \nonumber\\
		&\phantom{=}\qquad + 2\varepsilon\|\bm{u}_2\|^2_{r,K} + \cdots + \varepsilon\|\bm{u}_N\|^2_{r,K}\Big)^{1/2} \|\bm{v}_h\|_{\bm{Q}} \nonumber\\
		&\lesssim \sqrt{\varepsilon}h^{\min\{r,k+1\}-1} \big(1+\mathcal{O}(\varepsilon)\big)\|\bm{v}_h\|_{\bm{Q}}.
	\end{align}
	To bound $\mathbb{I}_3$, using the Cauchy-Schwarz inequality,
	\begin{align}\label{eq:error_inq_bi_3}
	|\mathbb{I}_3|
	&\lesssim \sum_{K\in\mathcal{T}_h}\left\|\sqrt{\bm{Q}}(\bm{u}-\mathcal{I}_h\bm{u})\right\|_K\left\|\sqrt{\bm{Q}}\bm{v}_h\right\|_K 
	\lesssim \sum_{K\in\mathcal{T}_h}\left\|\sqrt{\bm{Q}}(\bm{u}-\mathcal{I}_h\bm{u})\right\|_K\|\bm{v}_h\|_{\bm{Q}} \nonumber\\
	&\lesssim h^{\min\{r,k+1\}} \left\|\sqrt{\bm{Q}}\bm{u}\right\|_{r,X}\|\bm{v}_h\|_{\bm{Q}} 
	\lesssim \sqrt{\varepsilon}h^{\min\{r,k+1\}} \|\bm{v}_h\|_{\bm{Q}}.
	\end{align}
	Therefore, by \eqref{eq:error_inq_bi_parts}, \eqref{eq:error_inq_bi_1}, \eqref{eq:error_inq_bi_2}, and \eqref{eq:error_inq_bi_3}, we have
	\begin{equation}\label{eq:error_inq_bi}
	\left|\mathfrak{a}_h(\bm{u}-\mathcal{I}_h\bm{u},\bm{v}_h)\right|\lesssim \sqrt{\varepsilon}\Big(h^{\min\{r,k+1\}-1} \big(1+\mathcal{O}(\varepsilon)\big) + h^{\min\{r,k+1\}} \Big) \|\bm{v}_h\|_{\bm{Q}}.
	\end{equation}
	Combining \eqref{eq:error_inq2}, \eqref{eq:error_inq_part_one} and \eqref{eq:error_inq_bi} leads to the stated error estimate.
\end{proof}

Re-scaling the $\|\cdot\|_{\bm{Q}}$ in \eqref{eq:est_total_general} by \Cref{lem:norms_rescale}, we have the following theorem.
\begin{theorem}[Uniform error estimate]\label{thm:uni_conv} 
	Assume that \eqref{as:rte_scale_1}--\eqref{as:rte_scale_2} and \Cref{assupt:reg_1} hold. Consider only $k\ge 1$. 
	The solution $\bm{u}_h$ of the DG method \eqref{eq:bilinear} for the spherical harmonic radiative transfer equation \eqref{eq:rte_scale_Pn_all} converges to $\bm{u}$ uniformly in $\varepsilon$ as $h\rightarrow 0$, and admits the following error estimate:
	\begin{equation*}
	\left\|\bm{u}-\bm{u}_h\right\|
	\lesssim h^{\min\{r,k+1\}-1} (1+\varepsilon+\varepsilon^2) + h^{\min\{r,k+1\}}(1+\varepsilon), \quad k\ge 1.
	\end{equation*}
\end{theorem}

\Cref{thm:uni_conv} combined with \Cref{thm:ang_est} and noting that
\begin{equation*}
\|u-\bm{m}^{\tran}\bm{u}_h\|_{L^2(X\times\mathbb{S})} \le \|u - \bm{m}^{\tran}\bm{u}\|_{L^2(X\times\mathbb{S})} + \|\bm{u}-\bm{u}_h\|
\end{equation*}
lead to the following error estimate between the solutions to the scaled RTE \eqref{eq:rte_scale_all} and the SH-DG method \eqref{eq:bilinear}.
\begin{corollary}\label{cor:main}
	Assume the conditions of Theorems~\ref{thm:ang_est} and \ref{thm:X} hold. Then, we have
	\begin{multline*}
	\|u-\bm{m}^{\tran}\bm{u}_h\|_{L^2(X\times\mathbb{S})}
	\lesssim h^{\min\{r,k+1\}-1} (1+\varepsilon+\varepsilon^2) + h^{\min\{r,k+1\}}(1+\varepsilon) \\
	 + N^{-q} \left( \|u\|_{L^2(X; H^q(\mathbb{S}))}+\sum_{i=1}^{3}\|\partial_i u\|_{L^2(X; H^q(\mathbb{S}))} \right).
	\end{multline*}
\end{corollary}

\subsubsection{Error analysis for constant elements} Consider the case of $k=0$. 
\begin{theorem}\label{thm:X_k=0}
	Assume that \eqref{as:rte_scale_1}--\eqref{as:rte_scale_2} and \Cref{assupt:reg_1} hold. For $k= 0$, the spherical harmonic DG method \eqref{eq:bilinear} admits the following error estimate:
	\begin{equation}\label{eq:est_total_general_k=0}
		|||\bm{u}-\bm{u}_h|||_h \lesssim (1+\varepsilon) h^{\min\{r,1\}-1/2} + \sqrt{\varepsilon} h^{\min\{r,1\}},
	\end{equation}
	and
	\begin{equation}\label{eq:est_total_k=0}
		\left\|\bm{u}-\bm{u}_h\right\|
		\lesssim (\varepsilon^{-1/2} + \varepsilon^{1/2} + \varepsilon^{3/2}) h^{\min\{r,1\}-1/2} + (1 + \varepsilon) h^{\min\{r,1\}}, \quad k = 0.
	\end{equation}
\end{theorem}
\begin{proof}
It follows from the {G}alerkin orthogonality \eqref{eq:go}, the definition of $|||\cdot|||_h$ norm in \eqref{def:norm}, and Lemma \ref{lem:stab_0} that
\begin{equation}\label{eq:error_inq3}
	|||\bm{u}-\bm{u}_h|||_h \lesssim |||\bm{u} - \mathcal{I}_h\bm{u}|||_h + \sup_{\bm{v}_h\in \bm{V}_h}\frac{\mathfrak{a}_h(\bm{u}-\mathcal{I}_h\bm{u},\bm{v}_h)}{|||\bm{v}_h |||_h}.
\end{equation}
Note that for $k=0$, $[\![\mathcal{I}_h\bm{u}]\!]\big|_{\partial K}\neq 0$ since the interpolant cannot be generated by edge points. Therefore, $[\![\bm{u} - \mathcal{I}_h\bm{u}]\!]\big|_{\partial K}\neq 0$. Instead, we have
	\begin{align*}
		|||\bm{u} - \mathcal{I}_h\bm{u}|||^2_h &=\frac{1}{4}\sum_{K\in\mathcal{T}_h} \Big(|\bn_K|\cdot\bm{D}\,[\![\bm{u} - \mathcal{I}_h\bm{u}]\!], [\![\bm{u} - \mathcal{I}_h\bm{u}]\!]\Big)_{\partial K} \\
		&\phantom{=}\quad + \left(\bm{Q}(\bm{u} - \mathcal{I}_h\bm{u}),\bm{u} - \mathcal{I}_h\bm{u}\right) \nonumber\\ 
		&\lesssim h^{2\min\{r,1\}-1}\|\bm{u}\|^2_{r,X} + h^{2\min\{r,1\}} \left\|\sqrt{\bm{Q}}\bm{u}\right\|^2_{r,X}, \\
		&\lesssim (1+\varepsilon)^2 h^{2\min\{r,1\}-1} + \varepsilon h^{2\min\{r,1\}}, 
\end{align*}
from \Cref{lem:interp,assupt:reg_1}, i.e.,
\begin{equation}\label{eq:error_inq3_k=0}
	|||\bm{u} - \mathcal{I}_h\bm{u}|||_h  \lesssim (1+\varepsilon) h^{\min\{r,1\}-1/2} + \sqrt{\varepsilon} h^{\min\{r,k+1\}}.
\end{equation}
For the term $\mathfrak{a}_h(\cdot,\cdot)$, we have
$\mathfrak{a}_h(\bm{u}-\mathcal{I}_h\bm{u},\bm{v}_h) = \mathbb{I}_1 +\mathbb{I}_2 + \mathbb{I}_3$, where
\begin{align*}
	\mathbb{I}_1 &:= \sum_{K\in\mathcal{T}_h} \sum_{i=1}^3 \frac{1}{2}\Big(n_i^K\overrightarrow{\bm{A}^{(i)}(\bm{u}-\mathcal{I}_h\bm{u})}, [\![\bm{v}_h]\!]\Big)_{\partial K} \\
	\mathbb{I}_2 &:= -\sum_{K\in\mathcal{T}_h} \sum_{i=1}^3 (\bm{A}^{(i)}(\bm{u}-\mathcal{I}_h\bm{u}),\partial_i \bm{v}_h)_K, \quad
	\mathbb{I}_3 := \sum_{K\in\mathcal{T}_h} \left(\bm{Q}(\bm{u}-\mathcal{I}_h\bm{u}),\bm{v}_h\right)_K.
\end{align*}
For the first term, since $\mathcal{Q}_i$, $i=1,2,3$ are orthogonal, we have
\begin{align}\label{eq:error_inq_bi_1_k=0}
	|\mathbb{I}_1| &\le  \frac{1}{2}\left(\sum_{K\in\mathcal{T}_h} \sum_{i=1}^3\Big(|n_i^K|\,|\varLambda|^{(i)}\overrightarrow{\mathcal{Q}_i^{\tran}(\bm{u}-\mathcal{I}_h \bm{u})},\overrightarrow{\mathcal{Q}_i^{\tran}(\bm{u}-\mathcal{I}_h\bm{u})}\Big)_{\partial K}\right)^{1/2} \nonumber\\
	&\phantom{\qquad\qquad\qquad\qquad}\cdot\left(\sum_{K\in\mathcal{T}_h} \sum_{i=1}^3\Big(\big|n_i^K\big|\,\bm{D}^{(i)} [\![\bm{v}_h]\!],[\![\bm{v}_h]\!]\Big)_{\partial K}\right)^{1/2} \nonumber\\
	&\lesssim \sum_{i=1}^3 h^{\min\{r,1\}-1/2}\left\|\mathcal{Q}_i^{\tran}\bm{u}\right\|_{r,X} |||\bm{v}_h|||_h = 3h^{\min\{r,1\}-1/2}\|\bm{u}\|_{r,X} |||\bm{v}_h|||_h \nonumber \\
	&\lesssim (1+\varepsilon) h^{\min\{r,1\}-1/2} |||\bm{v}_h|||_h,
\end{align}
where the last inequality is due to the assumption \eqref{eq:reg_assump_u}. Since $k=0$, we have $\partial_i \bm{v}_h=0$ for any $\bm{v}_h\in \bm{V}_h$. Therefore,
\begin{equation}\label{eq:error_inq_bi_2_k=0}
	\mathbb{I}_2=0.
\end{equation}
The third term can be handled similarly to \eqref{eq:error_inq_bi_3}:
\begin{equation}\label{eq:error_inq_bi_3_k=0}
	|\mathbb{I}_3| \lesssim \sqrt{\varepsilon}h^{\min\{r,1\}} |||\bm{v}_h|||_h.
\end{equation}
Combining \eqref{eq:error_inq_bi_1_k=0}, \eqref{eq:error_inq_bi_2_k=0}, and \eqref{eq:error_inq_bi_3_k=0} gives
\begin{equation}\label{eq:error_inq_bi_k=0}
	\left|\mathfrak{a}_h(\bm{u}-\mathcal{I}_h\bm{u},\bm{v}_h)\right|\lesssim \left((1+\varepsilon) h^{\min\{r,1\}-1/2} + \sqrt{\varepsilon}h^{\min\{r,1\}} \right)|||\bm{v}_h|||_h.
\end{equation}

The estimate \eqref{eq:est_total_general_k=0} follows by gathering and inserting the estimates for \eqref{eq:error_inq3_k=0} and \eqref{eq:error_inq_bi_k=0} obtained above into \eqref{eq:error_inq3}. By the definition of $|||\cdot|||_h$, \eqref{eq:est_total_k=0} can be directly deduced from \eqref{eq:est_total_general_k=0}.
\end{proof}

\begin{remark}
	It is well known that the DG method with piecewise constant approximations does not achieve the diffusion limit \cite{LMM1987}. This fact is reflected by the term $\varepsilon^{-1/2}$ in \eqref{eq:est_total_k=0} that tends to infinity as $\varepsilon\to 0$.
\end{remark}

\section{Error analysis for tensor product elements on Cartesian mesh}\label{sec:err_ana_optimal}
The framework for the error analysis developed in \Cref{sec:err_ana} is
applicable to fairly general settings. We obtain a uniform $\mathcal{O}(h^k)$ bound across all $\varepsilon\in(0,1)$ when local polynomials of degree $k\ge 1$ are employed. However, for tensor product elements on Cartesian meshes of dimensions one, two, or three, optimal convergence results can be derived. Here, we consider the one-dimensional slab geometry model (since its error analysis has a unique approach different from the multidimensional cases) and the two-dimensional plane-parallel model \cite{agoshkov1998boundary,Modest} to simplify the analysis, which are briefly introduced below. Note that these reduced models correspond to a three-dimensional problem with certain symmetries. Since it is already known that DG methods do not perform well when $k = 0$, we focus only on the case $k \ge 1$.

\textbf{One-dimensional slab geometry problems.} 
In slab geometry, the RTE \eqref{eq:rte_scale_all} can reduce to the following form
\begin{subequations}\label{eq:1d}
	\begin{align}
		\mu\frac{\partial u}{\partial z}+\frac{\sigma_{\mathrm{t}}}{\varepsilon} u &= \frac{1}{2}\left(\frac{\sigma_{\mathrm{t}}}{\varepsilon}-\varepsilon\sigma_{\mathrm{a}}\right) \int^{1}_{-1} u(z,\mu')\ud \mu' +  f,\quad \forall z\in X=I:=(0,1),\\
		u(0,\mu)&=u(1,\mu) \quad \forall \mu\in [-1,1],
	\end{align}
\end{subequations}
where $\mu\in[-1,1]$ is the $z$-coordinate of $\bomega$, $u=u(z,\mu)$, $f=f(z,\mu)$, and $\ud \mu'$ is the Lebesgue measure on $(-1,1)$.
We set $\bm{A}^{(1)}=\bm{A}^{(2)}=\bm{0}$ and therefore $\bm{A}\cdot\nabla=\bm{A}^{(3)}\partial_z$, and all the notation and formulas in Sections~\ref{sec:SH_PN} and \ref{sec:DG} can be kept. Especially, \eqref{eq:rte_scale_Pn} takes the following form:
\begin{equation*}\label{eq:rte_scale_Pn_1D}
	\bm{A}^{(3)}\partial_3\bm{u} + \varepsilon\sigma_{\mathrm{a}}\bm{u} + \left(\frac{\sigma_{\mathrm{t}}}{\varepsilon}-\varepsilon\sigma_{\mathrm{a}}\right) \bm{R}\bm{u} = \varepsilon \bm{f}.
\end{equation*}

\textbf{2-D plane-parallel model.} We consider the 2-dimensional case $X=(0,1)^2$. The angular variable $\bm{\omega}$ is determined by polar angle $\theta\in[0,\pi]$ and azimuth $\varphi\in[0,2\pi)$ by use of the standard spherical coordinates, i.e.~$\bm{\omega}=\begin{bmatrix} \sin\theta\cos\varphi & \sin\theta\sin\varphi & \cos\theta \end{bmatrix}^{\tran}$. Set $\mu=\cos\theta$, and the RTE \eqref{eq:rte_scale_all} takes the form
\begin{subequations}\label{eq:2d}
	\begin{align}
		\sqrt{1-\mu^2}\cos\varphi\frac{\partial u}{\partial x} + \sqrt{1-\mu^2}\sin\varphi\frac{\partial u}{\partial y} + \frac{\sigma_{\mathrm{t}}}{\varepsilon}u
		&=\left(\frac{\sigma_{\mathrm{t}}}{\varepsilon}-\varepsilon\sigma_{\mathrm{a}}\right)\bar{u}+\varepsilon f,\\
		u(0,y,\mu,\varphi)&=u(1,y,\mu,\varphi), y\in(0,1), \\ 
		u(x,0,\mu,\varphi)&=u(x,1,\mu,\varphi), x\in(0,1),
	\end{align}
\end{subequations}
where $\mu=\cos\theta\in[-1,1]$, $\varphi\in [0,2\pi)$, $u=u(x,y,\mu,\varphi)$, and $f=f(x,y)$.
Then \eqref{eq:rte_scale_Pn} can be rewritten as
\begin{equation*}\label{eq:rte_scale_Pn_2D}
	\bm{A}^{(1)}\partial_1\bm{u} + \bm{A}^{(2)}\partial_2\bm{u} + \varepsilon\sigma_{\mathrm{a}}\bm{u} + \left(\frac{\sigma_{\mathrm{t}}}{\varepsilon}-\varepsilon\sigma_{\mathrm{a}}\right) \bm{R}\bm{u} = \varepsilon \bm{f}.
\end{equation*}

\subsection{Error analysis in spatial discretization with upwind flux for one-dimensional reduced radiative transfer equations}
Let us first introduce a special Radau projection $\mathcal{R}$, defined on an interval $J$, which will be very useful in the optimal error analysis.
The Radau projection is defined as follows: Given a value of $\mu$, the simplex $J$ has a unique outflow point $x^{\mathrm{out}}_J$ based on the sign of $\mu$, and, if $k\ge 1$,
\begin{subequations}\label{eq:Radau_proj}
	\begin{align}
		(u-\mathcal{R} u,v)_J &=0 \quad \forall v\in \mathbb{P}_{k-1}(J), \label{eq:Radau_proj_a}\\
		u(x^{\mathrm{out}}_{J})-\mathcal{R} u(x^{\mathrm{out}}_{J}) &=0. \label{eq:Radau_proj_b}
	\end{align}
\end{subequations}
In fact, if $\mu>0$, we define $x^{\mathrm{out}}_{J}$ to be the right boundary point of $J$; If $\mu<0$, $x^{\mathrm{out}}_{J}$ is then the left boundary point of $J$; If $\mu=0$, $x^{\mathrm{out}}_{J}$ could be either the right or the left boundary point of $J$, or $\mathcal{R}$ can be defined as the usual $L_2$-orthogonal projection in this case, since the advection terms vanish.

The following lemma is known in \cite[Lemma 2.1]{CDG2008_SINUM}.
\begin{lemma}\label{lem:Radau_interp}
	For all $u\in H^{r}(J)$ and $J\in\mathcal{T}_h $, we have
	\begin{equation*}
		\|u-\mathcal{R} u\|_{J}\le C h^{\min \{r,k+1\}}\|u\|_{r,J},
	\end{equation*}
	where $C$ depends only on $k$.
\end{lemma}

We have the following optimal spatial error estimate for the DG approximation in the one-dimensional slab geometry setting.
\begin{theorem}\label{thm:X_1d}
	If \eqref{as:rte_scale_1}--\eqref{as:rte_scale_2} hold and $\bm{u}$ satisfies \Cref{assupt:reg_1}, the spherical harmonic DG method \eqref{eq:bilinear} with the upwind flux for the one-dimensional slab geometry problem \eqref{eq:1d} admits the following error estimate for $k\geq 1$:
	\begin{equation*}
	\|\bm{u}-\bm{u}_h\|_I \lesssim (1+\varepsilon)h^{\min\{r,k+1\}}.
	\end{equation*}
\end{theorem}
\begin{proof}
	To derive the error estimate, we first observe that since $\mathcal{R}\bm{u}\in \bm{V}_h$, the inequality \eqref{eq:error_inq1} holds in terms of the norm $\|\cdot\|_{\bm{Q}}$:
	\begin{equation}\label{eq:error_inq2_1d}
	\|\bm{u}-\bm{u}_h\|_{\bm{Q}} \lesssim \|\bm{u} - \mathcal{R}\bm{u}\|_{\bm{Q}} + \sup_{\bm{v}_h\in \bm{V}_h}\frac{\mathfrak{a}_h(\bm{u}-\mathcal{R}\bm{u},\bm{v}_h)}{\|\bm{v}_h \|_{\bm{Q}}}.
	\end{equation}
	We estimate each term in the right-hand side of the above inequality.
	
	For the first term, owning to the \Cref{lem:Radau_interp} and \Cref{assupt:reg_1}, we have
	\begin{equation}\label{eq:error_inq3_1d}
	\|\bm{u}-\mathcal{R}\bm{u}\|_{\bm{Q}} \lesssim \sqrt{\varepsilon} h^{\min\{r,k+1\}}.
	\end{equation}
	For the second term, similar to the proof of \Cref{thm:X}, we have
	\begin{equation*}\label{eq:error_inq_bi_parts_1d}
	\mathfrak{a}_h(\bm{u}-\mathcal{R}\bm{u},\bm{v}_h) = \mathbb{I}_1 +\mathbb{I}_2 + \mathbb{I}_3,
	\end{equation*}
	where
	\begin{alignat*}{2}
	\mathbb{I}_1 &:=  \sum_{J\in\mathcal{T}_h}  \Big(n_3^J\overrightarrow{\bm{A}^{(3)}(\bm{u}-\mathcal{R}\bm{u})}, \bm{v}_h^{-}\Big)_{\partial J}, & &(n_3=\pm 1),\\
	\mathbb{I}_2 &:=  -\sum_{J\in\mathcal{T}_h}  (\bm{A}^{(3)}(\bm{u}-\mathcal{R}\bm{u}),\partial_z \bm{v}_h)_J, 
	&\quad\mathbb{I}_3 &:=  \sum_{J\in\mathcal{T}_h} \left(\bm{Q}(\bm{u}-\mathcal{R}\bm{u}),\bm{v}_h\right)_J.
	\end{alignat*}
	Since the upwind flux is assumed, for the term $\mathbb{I}_1$, due to \eqref{eq:A_sym}, we have
	\begin{align*} 
		\Big(n_3^J\overrightarrow{\bm{A}^{(3)}(\bm{u}-\mathcal{R}\bm{u})}, \bm{v}_h^{-}\Big)_{\partial J} &= \Big(n_3^J\mathcal{Q}_3 \varLambda^{(3)} \overrightarrow{\mathcal{Q}_3^{\tran}(\bm{u}-\mathcal{R}\bm{u})}, \bm{v}_h^{-}\Big)_{\partial J}\\
		&= \Big(n_3^J \varLambda^{(3)} \overrightarrow{\mathcal{Q}_3^{\tran}(\bm{u}-\mathcal{R}\bm{u})}, \mathcal{Q}_3^{\tran}\bm{v}_h^{-}\Big)_{\partial J}\\
		[\text{Set }\bm{w}=\mathcal{Q}_3^{\tran}\bm{u}]
		&= \Big(n_3^J \varLambda^{(3)} \overrightarrow{\bm{w} - \mathcal{R}\bm{w}}, \,\mathcal{Q}_3^{\tran}\bm{v}_h^-\Big)_{\partial J}.	
	\end{align*}
	Since $\varLambda^{(3)}$ is a diagonal matrix, by use of the property \eqref{eq:Radau_proj_b}, we infer  
	$\mathbb{I}_1=0$. Noting $\partial_z \bm{v}_h\in \mathbb{P}_{k-1}(J)$ and the property \eqref{eq:Radau_proj_a}, we get $\mathbb{I}_2=\sum_{J\in\mathcal{T}_h}\left(\bm{u}-\mathcal{R}\bm{u},\bm{A}^{(3)}\partial_z \bm{v}_h\right)_J=0$. By Lemma~\ref{lem:Radau_interp}, we have $|\mathbb{I}_3|\lesssim \sqrt{\varepsilon} h^{\min\{r,k+1\}} \|\bm{v}_h\|_{\bm{Q}}$. Combining all the bounds that have been derived gives
	\begin{equation}\label{eq:error_inq_bi_1d}
	\left|\mathfrak{a}_h(\bm{u}-\mathcal{R}\bm{u},\bm{v}_h)\right|\lesssim \sqrt{\varepsilon} h^{\min\{r,k+1\}} \|\bm{v}_h\|_{\bm{Q}}.
	\end{equation}
	By inserting \eqref{eq:error_inq3_1d} and \eqref{eq:error_inq_bi_1d} into \eqref{eq:error_inq2_1d} and noting \Cref{lem:norms_rescale}, we infer
	\begin{equation*}
		\|\bm{u}-\bm{u}_h\|_I \lesssim \left(\sqrt{\varepsilon}+\frac{1}{\sqrt{\varepsilon}}\right)\|\bm{u}-\bm{u}_h\|_{\bm{Q}} \lesssim (1+\varepsilon) h^{\min\{r,k+1\}},
	\end{equation*}
	which completes the proof.
\end{proof}

\subsection{Error analysis in spatial discretization with upwind flux for tensor product polynomials on rectangular elements of Cartesian mesh} 
The optimal estimate also holds in the two- and three-dimensional cases when the tensor product polynomials on rectangular elements of Cartesian mesh are employed. We focus here only on the two-dimensional case ($d=2$). The extension of our analysis to the case $d=3$ is straightforward.

On a rectangle $K=J_1\times J_2$, for $w\in C^0(\overline{K})$, we define
\begin{equation*}
	\Pi w:=\mathcal{R}_1\otimes\mathcal{R}_2 w
\end{equation*}
with the subscripts of $\mathcal{R}$ indicating the application of the one-dimensional operators $\mathcal{R}$ in \eqref{eq:Radau_proj} with respect to the corresponding variable. The following approximation result is known (see, e.g., \cite[Lemma 3.2]{CKPS2001}).
\begin{lemma}\label{lem:Radau_interp_2d}
	For all $w\in H^{r}(J)$ and $J\in\mathcal{T}_h $, we have, for $k\ge 1$, 
	\begin{equation}
		\|w-\Pi w\|_K\le C h^{\min \{r,k+1\}}\|w\|_{r,K}, 
	\end{equation}
	where $C$ depends only on $k$.
\end{lemma}

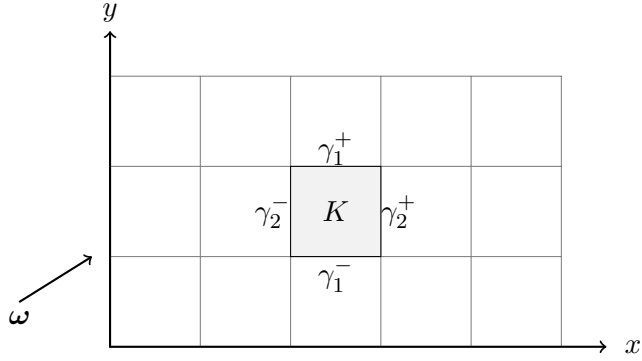
\begin{figure}[!hbp]
	\centering
	\begin{tikzpicture}[scale=1.2,cap=round]
		\draw[step=1cm,gray,thin] (0,0) grid (5,3);
		\filldraw[fill=gray!10, draw=black] (2,1) rectangle (3,2);
		\draw[thick,->] (0,0) -- (5.5,0);
		\node[right] at (5.6,0.0) {\large$x$};
		\draw[thick,->] (0,0) -- (0,3.5);
		\node[above] at (0,3.5) {\large$y$};
		\node[black,centered] at (2.5,1.5) {\large${K}$};
		\node[black,centered] at (2.5,0.8) {\large$\gamma_1^-$};
		\node[black,centered] at (2.5,2.2) {\large$\gamma_1^+$};
		\node[black,centered] at (1.8,1.5) {\large$\gamma_2^-$};
		\node[black,centered] at (3.2,1.5) {\large$\gamma_2^+$};
		\draw[thick,->] (-1,0.5) -- (-0.2,1.0);
		\node[below] at (-1,0.5) {\large$\bomega$};
	\end{tikzpicture}
	\caption{The Cartesian mesh $\mathcal{T}_h$ and the vector $\bomega$ used to define the numerical fluxes.}\label{fig:1}
\end{figure}
Consider a cuboid $K\in \mathcal{T}_h$ with four edges $\gamma_1^-$, $\gamma_1^+$, $\gamma_2^-$, and $\gamma_2^+$ as in \Cref{fig:1}. Given a direction vector $\bomega$, the boundary of $K$ is partitioned into two parts: $\partial K=\gamma^-\cup\gamma^+$, where $\gamma^-=\gamma_1^-\cup\gamma_2^-$ is the incoming boundary and $\gamma^+=\gamma_1^+\cup\gamma_2^+$ is the outgoing one. Note that the sets $\gamma^-$ and $\gamma^+$ depend on the choice of $\bomega$. Define
\begin{subequations}\label{eq:Z_K}
	\begin{align}
		Z_{K,1}(u,v_h)&= (u-\mathcal{R}_2 u, v_h^-)_{\gamma_2^+} - (u-\mathcal{R}_2 u, v_h^-)_{\gamma_2^-} - (u-\Pi u,\partial_1 v_h)_K, \\
		Z_{K,2}(u,v_h)&=(u-\mathcal{R}_1 u, v_h^-)_{\gamma_1^+} - (u-\mathcal{R}_1 u, v_h^-)_{\gamma_1^-} - (u-\Pi u,\partial_2 v_h)_K.
	\end{align}
\end{subequations}
Note that $v_h^-$ indicates the inside value of $v_h$ with respect to $K$. 
When vector arguments are employed in $Z_{K,i}$, $i=1,2$, the right-hand sides of \eqref{eq:Z_K} are understood in terms of \eqref{eq:vector_inner}. 
The following superconvergence result \cite{CKPS2001}, essentially due to LeSaint and Raviart \cite{LR1974}, plays an essential role in obtaining the optimal estimate. 
\begin{lemma}\label{lem:Z_K}
	Let $Z_{K,1}$ and $Z_{K,2}$ be defined by \eqref{eq:Z_K}. Assume $u\in H^r(K)$ and $v_h\in \mathbb{Q}_k(K)$. Then we have for $r>0$,
	\begin{equation*} 
		|Z_{K,i}(u,v_h)|\le C h^{\min\{r,k+1\}} \|u\|_{r,K}\|v _h\|_K, \quad i=1,2,
	\end{equation*}
	where the constant $C$ depends only on $r$ and $k$.
\end{lemma}

\begin{theorem}\label{thm:X_2d}
	Assume all the elements in $\mathcal{T}_h$ are rectangles and tensor product polynomials of degree at most $k$ are used.   
	If \eqref{as:rte_scale_1}--\eqref{as:rte_scale_2} hold and $\bm{u}$ satisfies \Cref{assupt:reg_1}, the spherical harmonic DG method \eqref{eq:bilinear} with the upwind flux for the two-dimensional plane-parallel problem \eqref{eq:2d} admits the following error estimate for $k\ge 1$:
	\begin{equation*}
		\|\bm{u}-\bm{u}_h\| \lesssim \left(1+\mathcal{O}(\varepsilon)\right)h^{\min\{r,k+1\}}.
	\end{equation*}
\end{theorem}
\begin{proof}
	To derive the error estimate, we first observe that since $\Pi\bm{u}\in \bm{V}_h$ and $\mathfrak{a}_h(\bm{v}_h,\bm{v}_h) = |||\bm{v}_h|||^2_h\ge \|\bm{v}_h\|_{\bm{Q}}^2$, the inequality \eqref{eq:error_inq1} holds in term of the norm $\|\cdot\|_{\bm{Q}}$:
	\begin{equation}\label{eq:error_inq2_2d}
		\|\bm{u}-\bm{u}_h\|_{\bm{Q}} \lesssim \|\bm{u} - \Pi \bm{u}\|_{\bm{Q}} + \sup_{\bm{v}_h\in \bm{V}_h}\frac{\mathfrak{a}_h(\bm{u}-\Pi \bm{u},\bm{v}_h)}{\|\bm{v}_h\|_{\bm{Q}}}.
	\end{equation}
	We estimate each term on the right-hand side of the above inequality.
	
	For the first term, owning to \Cref{lem:Radau_interp_2d} and \Cref{assupt:reg_1}, we have
	\begin{equation}\label{eq:error_inq3_2d}
		\|\bm{u}-\Pi \bm{u}\|_{\bm{Q}} \lesssim \sqrt{\varepsilon} h^{\min\{r,k+1\}}.
	\end{equation}
	For the second term, we have, similar to the proof of \Cref{thm:X},
	\begin{equation}\label{eq:error_inq_bi_parts_2d}
		\mathfrak{a}_h(\bm{u}-\Pi \bm{u},\bm{v}_h) = \sum_{K\in\mathcal{T}_h}\mathbf{Z}_K(\bm{u},\bm{v}_h) + \mathbb{I}_3,
	\end{equation}
	where $\mathbb{I}_3 :=  \sum_{K\in\mathcal{T}_h} \left(\bm{Q}(\bm{u}-\Pi \bm{u}),\bm{v}_h\right)_K$, and 
	\begin{equation*}
		\mathbf{Z}_K(\bm{u},\bm{v}_h):= -\sum_{i=1}^2 \left(\bm{A}^{(i)}(\bm{u}-\Pi \bm{u}),\partial_i \bm{v}_h\right)_K + \sum_{i=1}^2 \left(n_i^K\overrightarrow{\bm{A}^{(i)}(\bm{u}-\Pi \bm{u})}, \bm{v}_h^-\right)_{\partial K}.\label{eq:ZZ_K}
	\end{equation*}
	By \Cref{lem:Radau_interp_2d} and \Cref{assupt:reg_1}, we infer
	\begin{equation}\label{eq:est_I_3}
		|\mathbb{I}_3|\lesssim \sqrt{\varepsilon}h^{\min\{r,k+1\}} \|\bm{v}_h\|_{\bm{Q}}.
	\end{equation}
	Next, we estimate $\mathbf{Z}_K(\bm{u},\bm{v}_h)$ and proceed as follows: 
	\begin{align*}
		\mathbf{Z}_K(\bm{u},\bm{v}_h)&= \sum_{i=1}^2\left( \big(n_i^K\overrightarrow{\bm{A}^{(i)}(\bm{u}-\Pi \bm{u})}, \bm{v}_h^-\big)_{\partial K} - \big(\bm{A}^{(i)}(\bm{u}-\Pi \bm{u}),\partial_i \bm{v}_h\big)_K\right) \\
		&= \sum_{i=1}^2 \left(\big(n_i^K\mathcal{Q}_i \varLambda^{(i)} \overrightarrow{\mathcal{Q}_i^{\tran}(\bm{u}-\Pi \bm{u})}, \bm{v}_h^-\big)_{\partial K} - \big(\mathcal{Q}_i \varLambda^{(i)} \mathcal{Q}_i^{\tran}(\bm{u}-\Pi \bm{u}),\partial_i \bm{v}_h\big)_K \right) \\
		&= \sum_{i=1}^2 \left(\big(n_i^K\varLambda^{(i)} \overrightarrow{\mathcal{Q}_i^{\tran}(\bm{u}-\Pi \bm{u})}, \mathcal{Q}_i^{\tran}\bm{v}_h^-\big)_{\partial K} - \big(\varLambda^{(i)} \mathcal{Q}_i^{\tran}(\bm{u}-\Pi \bm{u}),\partial_i \mathcal{Q}_i^{\tran}\bm{v}_h\big)_K \right)\\
		\big[&\text{Set }\bm{w}_i=\mathcal{Q}_i^{\tran}\bm{u}\big]\\
		&= \sum_{i=1}^2 \left(n_i^K \overrightarrow{\varLambda^{(i)}(\bm{w}_i-\Pi \bm{w}_i)}, \mathcal{Q}_i^{\tran}\bm{v}_h^-\right)_{\partial K} -\sum_{i=1}^2 \left( \varLambda^{(i)} (\bm{w}_i-\Pi \bm{w}_i),\partial_i \mathcal{Q}_i^{\tran}\bm{v}_h\right)_K\\
		&= \sum_{i=1}^2 \varLambda^{(i)}Z_{K,i}(\bm{w}_i,\mathcal{Q}_i^{\tran}\bm{v}_h). 
	\end{align*}
	The last equality is due to the fact that $\overrightarrow{\varLambda^{(i)}\bm{w}_i}|_{\gamma_i^\pm}=\varLambda^{(i)}\bm{w}_i|_{\gamma_i^\pm}$ and $\overrightarrow{\varLambda^{(i)}\Pi\bm{w}_i}|_{\gamma_i^-} =\varLambda^{(i)}\mathcal{R}_i(\bm{w}_i|_{\gamma_i^-})$ since the edges of the rectangle $K\in\mathcal{T}_h$ are parallel to the $x/y$-axes and $\varLambda_i$ is diagonal. 
	By \Cref{lem:Z_K} and $\left\|\mathcal{Q}_i^{\tran}\bm{v}_h\right\|_K=\|\bm{v}_h\|_K$ since $\mathcal{Q}_i$, $i=1,2$ are orthogonal matrices, we infer that 
	\begin{equation}\label{eq:ZZ_est}
		\left|\mathbf{Z}_K(\bm{u},\bm{v}_h)\right|\le C h^{\min\{r,k+1\}} \|\bm{w}\|_{r,K}\left\|\mathcal{Q}_i^{\tran}\bm{v}_h\right\|_K\le C h^{\min\{r,k+1\}} \|\bm{u}\|_{r,K}\|\bm{v}_h\|_K.
	\end{equation}
	Define $\bm{Q}'= \diag\left(\sqrt{\varepsilon}/\sqrt{\sigma_{\mathrm{t}}},\sqrt{\varepsilon}/\sqrt{\sigma_{\mathrm{t}}}+(\sqrt{\varepsilon\sigma_{\mathrm{a}}})^{-1}, \sqrt{\varepsilon}/\sqrt{\sigma_{\mathrm{t}}}, \cdots, \sqrt{\varepsilon}/\sqrt{\sigma_{\mathrm{t}}}\right)$. A direct calculation shows that for $i=1,2$,
	\begin{equation*}
		\sqrt{\bm{Q}}^{-1}\bm{A}^{(i)}=\bm{A}^{(i)}\bm{Q}'-\begin{bmatrix}
			0 	& \frac{\sqrt{\varepsilon}}{\sqrt{\sigma_{\mathrm{t}}}}\bm{A}^{(i)}_{0,1} 	& 0   &\dots 		& 0  \\
			0 	& 0 	  	& 0		& \dots 		& 0\\
			0 	& \sqrt{\varepsilon\sigma_{\mathrm{a}}}^{-1}\bm{A}^{(i)}_{2,1}  	& 0 			& \dots		& 0 \\
			0 	& 0		& 0			& \cdots	& 0 \\
			\vdots 		& \vdots 	& \vdots	& \ddots	& \vdots \\
			0 	& 0		& 0			& \cdots	& 0
		\end{bmatrix},
	\end{equation*} 
	and therefore
	\begin{multline*}
		\sqrt{\bm{Q}}^{-1}\bm{A}^{(i)}\bm{u}=\bm{A}^{(i)}\bm{Q}'\bm{u}\\  -\diag(\sqrt{\varepsilon}/\sqrt{\sigma_{\mathrm{t}}},0,\sqrt{\varepsilon\sigma_{\mathrm{a}}}^{-1},0,\cdots,0)\bm{A}^{(i)}\begin{bmatrix} \mathbf{0} & \bm{u}_1^\tran & \mathbf{0} &\cdots & \mathbf{0} \end{bmatrix}^{\tran}.
	\end{multline*}
	We deduce 
	\begin{align*}
		\sum_{K\in\mathcal{T}_h}&\big|\mathbf{Z}_K(\bm{u},\bm{v}_h)\big|
		=\sum_{K\in\mathcal{T}_h}\big|\sqrt{\bm{Q}}^{-1}\mathbf{Z}_K(\bm{u},\sqrt{\bm{Q}}\bm{v}_h)\big|\\
		&= \sum_{K\in\mathcal{T}_h} \bigg|\mathbf{Z}_K(\bm{Q}'\bm{u},\sqrt{\bm{Q}}\bm{v}_h)\\
			&\qquad\qquad -\mathbf{Z}_K(\begin{bmatrix} \mathbf{0} & \bm{u}_1^\tran & \mathbf{0} &\cdots & \mathbf{0} \end{bmatrix}^{\tran},\diag(\sqrt{\varepsilon},0,\sqrt{\varepsilon}^{-1},0,\cdots,0)\sqrt{\bm{Q}}\bm{v}_h)\bigg| \\
		&\le \sum_{K\in\mathcal{T}_h} \big|\mathbf{Z}_K(\bm{Q}'\bm{u},\sqrt{\bm{Q}}\bm{v}_h)\big| \\
		&\quad\quad +\sum_{K\in\mathcal{T}_h}\big|\mathbf{Z}_K(\begin{bmatrix} \mathbf{0} & \bm{u}_1^\tran & \mathbf{0} &\cdots & \mathbf{0} \end{bmatrix}^{\tran},\diag(\sqrt{\varepsilon},0,\sqrt{\varepsilon}^{-1},0,\cdots,0)\sqrt{\bm{Q}}\bm{v}_h)\big|.		
	\end{align*}
	Note that the estimate \eqref{eq:ZZ_est} holds for any $\bm{u}\in [H^r(X)]^L$ and $\bm{v}_h\in V_h$. Therefore, 
	\begin{align*}
		\sum_{K\in\mathcal{T}_h} |\mathbf{Z}_K(\bm{Q}'\bm{u},\sqrt{\bm{Q}}\bm{v}_h)|&\le  \sum_{K\in\mathcal{T}_h} h^{\min\{r,k+1\}} \|\bm{Q}'\bm{u}\|_{r,K}\|\sqrt{\bm{Q}}\bm{v}_h\|_K \\
		&\lesssim \sqrt{\varepsilon}h^{\min\{r,k+1\}} \big(1+\mathcal{O}(\varepsilon)\big)\|\bm{v}_h\|_{\bm{Q}}. 
	\end{align*}
	Finally, 
	\begin{align*}
		\sum_{K\in\mathcal{T}_h}\bigg|\mathbf{Z}_K&\left(\begin{bmatrix} 0 & \bm{u}_1^\tran & \mathbf{0} &\cdots & \mathbf{0} \end{bmatrix}^{\tran},\diag(\sqrt{\varepsilon},0,\sqrt{\varepsilon}^{-1},0,\cdots,0)\sqrt{\bm{Q}}\bm{v}_h\right)\bigg| \\
		&\le (\sqrt{\varepsilon}+\sqrt{\varepsilon}^{-1})\sum_{K\in\mathcal{T}_h}h^{\min\{r,k+1\}}\|\bm{u}_1\|_{r,K} \|\sqrt{\bm{Q}}\bm{v}_h\|_K \\
		&\lesssim \sqrt{\varepsilon}(1+\varepsilon)h^{\min\{r,k+1\}} \|\bm{v}_h\|_{\bm{Q}}.
	\end{align*}
	Hence,
	\begin{equation}\label{eq:est_Z_K}
		\sum_{K\in\mathcal{T}_h}\big|\mathbf{Z}_K(\bm{u},\bm{v}_h)\big|\lesssim \sqrt{\varepsilon}(1+\varepsilon)h^{\min\{r,k+1\}} \|\bm{v}_h\|_{\bm{Q}}.
	\end{equation}
	With the bounds \eqref{eq:est_I_3} and \eqref{eq:est_Z_K}, we conclude from \eqref{eq:error_inq_bi_parts_2d} that
	\begin{equation}\label{eq:error_inq_bi_2d}
		\left|\mathfrak{a}_h(\bm{u}-\mathcal{R}\bm{u},\bm{v}_h)\right|\lesssim \sqrt{\varepsilon}(1+\varepsilon)h^{\min\{r,k+1\}} \|\bm{v}_h\|_{\bm{Q}}.
	\end{equation} 
	The result follows by inserting \eqref{eq:error_inq3_2d} and \eqref{eq:error_inq_bi_2d} into \eqref{eq:error_inq2_2d} and re-scaled by $\sqrt{\bm{Q}}$ with \Cref{lem:norms_rescale}.
\end{proof}

\section{Conclusions}\label{sec:conclusion}
In this paper, we analyze the convergence of a spherical harmonic DG scheme for scaled radiative transfer equations with isotropic scattering. We first prove that the spherical harmonic approximations for the angular variable converge uniformly with respect to $\varepsilon$. For sufficiently rich approximation spaces, we prove uniform convergence rates with respect to $\varepsilon$ for the DG scheme in the spacial variable. However, this convergence rate is in general not optimal. By employing the Radau projection and previous results for the DG method \cite{LR1974} for linear hyperbolic problems, we are able to further obtain the optimal and uniform convergence rate on Cartesian grids with tensor product polynomials of degree at least one.  

In future work, we hope to leverage the current analysis for more physically realistic scenarios that do not rely on the assumptions in \eqref{as:rte_scale}, but rather allow for arbitrarily thin and thick materials in the same problem.  In addition, problems with more realistic boundary conditions and less optimistic regularity assumptions will be considered.

\bibliographystyle{amsplain}
\bibliography{AP&unif_Conv_SH_DG_RTE}

\end{document}